\documentclass[11pt]{amsart}

\usepackage{amsmath, amsthm, amssymb}

\usepackage{lmodern}
\usepackage{palatino}                       % text font
\usepackage[bigdelims,vvarbb]{newpxmath}    % math font
\usepackage[scaled=0.95]{inconsolata}       % Sans font 
\usepackage[T1]{fontenc}                    % font encoding                    
\usepackage{comment}

\usepackage[abs]{overpic}		  % overpic automatically loads graphicx

\usepackage{xcolor}
\definecolor{indigo}{rgb}{0.29, 0.0, 0.51}  % custom colors
\usepackage[colorlinks, urlcolor=indigo, linkcolor=indigo, citecolor=indigo]{hyperref}

\usepackage[hcentering, top = 1.5in, total={5.9in, 8.2in}]{geometry}  % 1.4 vertical margin; 1.3 horizontal margin

\usepackage{systeme}
\usepackage{mathtools}

\theoremstyle{plain}
\newtheorem{theorem}{Theorem}
\newtheorem{corollary}[theorem]{Corollary}
\newtheorem{proposition}[theorem]{Proposition}
\newtheorem{lemma}[theorem]{Lemma}

\newtheorem{claim}{Claim}

\theoremstyle{definition}

\theoremstyle{remark}
\newtheorem{remark}[theorem]{Remark}

\numberwithin{theorem}{section}

\newcommand{\dfn}[1]{{\em #1}}        % definition
\newcommand{\R}{\mathbb{R}}           % the real numbers
\DeclareMathOperator{\bd}{\partial}   % boundary
\newcommand{\modp}[1]{\;(\!\!\!\!\!\!\mod #1)}      % mod for display math. \pmod is for inline math 
\makeatletter
\newcommand*\bigcdot{\mathpalette\bigcdot@{0.6}}
\newcommand*\bigcdot@[2]{\mathbin{\vcenter{\hbox{\scalebox{#2}{$\m@th#1\bullet$}}}}}
\makeatother

\DeclareMathOperator\tb{tb}                   % Thurston-Bennequin
\DeclareMathOperator\tbb{\overline {\tb}}     % maximum Thurston-Bennequin
\DeclareMathOperator\rot{rot}                 % rotation
\DeclareMathOperator\self{sl}                 % self linking
\DeclareMathOperator\selfb{\overline {\self}} % maximum self linking
\DeclareMathOperator\tw{tw}                   % twisting number
\DeclareMathOperator\twb{\overline {\tw}}      % maximum twisting number

\DeclareMathOperator{\Diff}{Diff}         % group of diffeomorphisms 
\DeclareMathOperator{\Cont}{Cont}         % group of contactomorphisms 

\DeclareFontFamily{U} {cmr}{}
\DeclareFontShape{U}{cmr}{m}{n}{
  <-6> cmr5
  <6-7> cmr6
  <7-8> cmr7
  <8-9> cmr8
  <9-10> cmr9
  <10-12> cmr8
  <12-> cmr9}{}
\DeclareSymbolFont{Xcmr} {U} {cmr}{m}{n}
\DeclareMathSymbol{\Phi}{\mathord}{Xcmr}{8}

\begin{document}

% title
\title[The contact mapping class group of lens spaces]{The contact mapping class group\\and rational unknots in lens spaces} 

\author{Hyunki Min}
\address{Department of Mathematics \\ University of California \\ Los Angeles, CA, United States}
\email{hkmin27@math.ucla.edu}

%\subjclass[2020]{57R17}

% abstract
\begin{abstract}
  We determine the contact mapping class group of the standard contact structures on lens spaces. To prove the main result, we use the one-parametric convex surface theory to classify Legendrian and transverse rational unknots in any tight contact structure on lens spaces up to Legendrian and transverse isotopy.
\end{abstract}

\maketitle
%\tableofcontents

%%%%%%%%%%%%%%%%%%%%%%%%%%%%%%%%%%%
\section{Introduction}
%%%%%%%%%%%%%%%%%%%%%%%%%%%%%%%%%%%

Ever since Eliashberg \cite{Eliashberg:S3} determined the homotopy type of the group of contactomorphisms of $(S^3,\xi_{std})$ relative to a point, there have been several studies on the group of contactomorphisms of various contact manifolds. For example, Gompf \cite{Gompf:S1xS2}   showed that there exists a contactomorphism of the standard contact structure on $S^1 \times S^2$ which is smoothly isotopic to the identity, but not contact isotopic to the identity. After that, there have been many similar results, see \cite{Bourgeois:exotic, DG:S1xS2, GG:exact, Gironella:highdim, MN:exotic} for examples. Also, there have been studies on higher homotopy groups, see \cite{BEM:overtwisted, Bourgeois:exotic, CS:higher, EM:R3, FG:higher, FMP:higher} for examples.

However, the contact mapping class group has not been determined for many contact manifolds so far. If we focus on closed manifolds, the contact mapping class group was only determined for the standard contact structure on $S^3$ by Eliashberg \cite{Eliashberg:S3}, the overtwisted contact structures on $S^3$ by Vogel \cite{Vogel:S3}, and the canonical contact structures on the unit cotangent bundles $U^*\Sigma_g$ for $g \geq 1$ and their cyclic covers by Giroux and Massot \cite{GM:bundles}. The main difficulty to study the contact mapping class group is that we need to understand the behavior of some Legendrian (or contact) submanifolds under Legendrian (or contact) isotopy, which has been poorly studied in other than $S^3$. According to the author's knowledge, the only known results on the classification of Legendrian knots outside of $S^3$ were the linear curves in any tight contact structure on $T^3$ by Ghiggini \cite{Ghiggini:Legendrian} and some knots and links in the standard contact structure on $S^1 \times S^2$ by Ding and Geiges \cite{DG:helix} and by Chen, Ding and Li \cite{CDL:knots1s2}. 

In Theorem~\ref{thm:cmcg-lens} and~\ref{thm:cmcg-s1s2}, we determine the contact mapping class group of the standard contact structures on lens spaces and $S^1 \times S^2$. The main ingredients are Theorem~\ref{thm:unknot-Legendrian} and~\ref{thm:unknot-transverse}, the classification of Legendrian and transverse rational unknots up to Legendrian and transverse isotopy. Once we have the classification, we can perturb a contactomorphism to fix a standard neighborhood of some Legendrian rational unknot. Then the problem reduces to determine the contact mapping class group of its complement, which is determined in Theorem~\ref{thm:cmcg-s1d2} by applying the one-parametric convex surface theory, in particular, Colin's isotopy discretization \cite{Colin:discretization}, and properties of bypasses studied by Honda \cite{Honda:bypass}, see Section~\ref{sec:bypass}

The main technique for Theorem~\ref{thm:unknot-Legendrian} and~\ref{thm:unknot-transverse} is again the application of one-parametric convex surface theory and various properties of bypasses, which were utilized to study Legendrian and transverse knots in contact structures on $S^3$ in \cite{CEM:cables, EMM:nonloose}. In Section~\ref{sec:unknot-lens}, we develop the technique in the setting of general lens spaces and keep track of Legendrian isotopy classes of rational unknots, see Proposition~\ref{prop:unknot}.

%--------------------------------------------------------------------------------
\subsection{The contact mapping class group of lens spaces}\label{sec:cmcg-intro}
%--------------------------------------------------------------------------------
We first review the basic notations. First, we assume every contactomorphism is a coorientation preserving one unless otherwise specified. We denote the group of contactomorphisms of a closed contact manifold $(M,\xi)$ by 
\[
  \Cont(M,\xi) = \text{the group of coorientation preserving contactomorphisms of $(M,\xi)$}.
\]
The \dfn{contact mapping class group of $(M,\xi)$} is defined to be the group of contact isotopy classes of contactomorphisms of $(M,\xi)$. We denote it by 
\[
  \pi_0(\Cont(M,\xi)) = \Cont(M,\xi)/\sim
\]
where $f \sim g$ if $f$ is contact isotopic to $g$.

Let $U$ be the unknot in $S^3$. For a pair of coprime integers $(p,q)$ satisfying $p > q > 0$, we define 
\[
  L(p,q) = S^3_{-p/q}(U).
\]
Now we are ready to state our first main result.

\begin{theorem}\label{thm:cmcg-lens}
  The contact mapping class group of $(L(p,q),\xi_{std})$ is 
  \[
    \pi_0(\Cont(L(p,q), \xi_{std})) = \begin{cases} 
      \mathbb{Z}_2 &\; q^2 \equiv 1\, \text{ and }\, q \not\equiv 1 \modp p, \\ 
      %\mathbb{Z}_2 &\; p \neq 2\, \text{ and }\, q \equiv -1 \modp p, \\ 
      %\mathbb{Z}_2 &\; q \not\equiv \pm1 \modp p\, \text{ and }\, q^2 \equiv 1 \modp p,\\
      1 &\; \text{otherwise}. \end{cases}
  \]
\end{theorem}

For the first two cases, the contact mapping class group is generated by a contactomorphism $\sigma$. See Section~\ref{sec:lens} for the definition of $\sigma$.

There is a quick application of Theorem~\ref{thm:cmcg-lens}. In the forthcoming paper with Baker, Etnyre, and Onaran \cite{BEMO:torus}, we classify Legendrian and transverse torus knots in the standard tight contact structures on lens spaces. Basically, we classify the knots \dfn{coarsely}, meaning that up to coorientation preserving contactomorphism which is smoothly isotopic to the identity. However, due to Theorem~\ref{thm:cmcg-lens}, we could improve the result up to Legendrian and transverse isotopy.

Comparing the contact mapping class group and the smooth mapping class group of lens spaces (Theorem~\ref{thm:mcg-lens}), we can make the following observation. 

\begin{corollary}\label{cor:inclusion}
  The induced map from the natural inclusion 
  \[
    i_*\colon \pi_0(\Cont(L(p,q), \xi_{std})) \to \pi_0(\Diff_+(L(p,q))) 
  \]
  is an isomorphism if and only if $q \equiv -1 \pmod p$. In particular, $i_*$ is injective but not surjective if $q \not\equiv -1 \pmod p$. 
\end{corollary}

The main reason for the failure of $i_*$ being surjective is that there exist two non-isotopic standard contact structures on $L(p,q)$ if and only if $q \not\equiv -1 \pmod p$. See Section~\ref{sec:lens} for more details about the standard contact structures on lens spaces. 

Since we can consider $S^1 \times S^2$ as $L(0,1)$, we also determine the contact mapping class group of the standard contact structure on $S^1 \times S^2$. We should mention that the main part of the proof was essentially done by Ding and Geiges \cite{DG:S1xS2}.

\begin{theorem}\label{thm:cmcg-s1s2}
  The contact mapping class group of $(S^1 \times S^2, \xi_{std})$ is 
  \begin{gather*}
    \pi_0(\Cont(S^1 \times S^2, \xi_{std})) = \mathbb{Z} \oplus \mathbb{Z}_2.
  \end{gather*}  
\end{theorem}

Recall that $\Diff_0(M)$ is the connected component of $\Diff_+(M)$ containing the identity. We can define a subgroup of $\Cont(M,\xi)$ as follows:  
\[
  \Cont_0(M,\xi) := \Cont(M,\xi) \cap \Diff_0(M).
\]
Ding and Geiges \cite{DG:S1xS2} proved that $\pi_0(\Cont_0(S^1 \times S^2, \xi_{std}))$ is isomorphic to $\mathbb{Z}$. From Corollary~\ref{cor:inclusion}, it is immediate that $\pi_0(\Cont_0(L(p,q),\xi_{std}))$ is trivial.

\begin{corollary}\label{cor:cont0}
  For every lens space $L(p,q)$, we have
  \[
    \pi_0(\Cont_0(L(p,q),\xi_{std})) = 1.
  \]
\end{corollary}

%-------------------------------------------------------------------------------------------------
\subsection{Legendrian and transverse rational unknots in lens spaces}\label{sec:unknot-intro}
%-------------------------------------------------------------------------------------------------
A \dfn{rational unknot} in a lens space is a core of a Heegaard torus. See Section~\ref{sec:mcg-lens} for more details, and also see Figure~\ref{fig:unknots} and \ref{fig:unknots-smooth} for (contact) surgery presentations of rational unknots in lens spaces. 

\begin{figure}[htbp]
  \vspace{0.2cm}
  \begin{overpic}{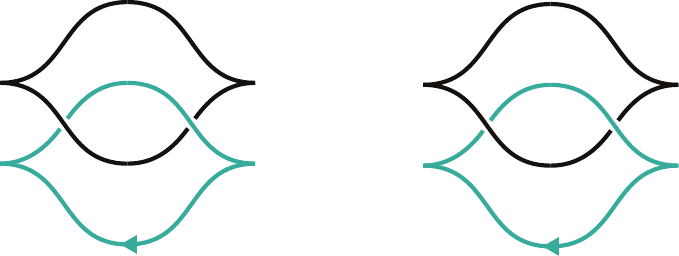}
    \put(130, 80){$\left(\frac{q-p}{q}\right)$}
    \put(130, 40){$L_1$}

    \put(336, 80){$\left(\frac{p'-p}{p'}\right)$}
    \put(336, 40){$L_2$}
  \end{overpic}
  \vspace{0.1cm}
  \caption{Two contact surgery presentations for the Legendrian rational unknots in tight contact structures on $L(p,q)$ with $\tbb_{\mathbb{Q}}$, the maximal rational Thurston--Bennequin number. Here, $p'/q'$ is the largest rational number satisfying $pq' - p'q = -1$.}
  \label{fig:unknots}
\end{figure}

Etnyre and Baker \cite{BE:rational} coarsely classified Legendrian rational unknots in any tight contact structure on lens spaces, see Theorem~\ref{thm:unknot-coarse}. Recall that the coarse classification is the classification up to contactomorphism which is smoothly isotopic to the identity. Also, Geiges and Onaran \cite{GO:unknots} coarsely classified non-loose Legendrian unknots in some lens spaces.  

However, to study the contact mapping class group, we need to understand the behavior of Legendrian (or contact) submanifolds under Legendrian (or contact) isotopy. Our second main result is the classification of Legendrian rational unknots in any tight contact structure on lens spaces up to Legendrian isotopy. 

\begin{theorem}\label{thm:unknot-Legendrian}
  Suppose $p > q > 0$ and $\xi$ is a tight contact structure on $L(p,q)$. Rational unknots in $\xi$ are Legendrian simple: there are Legendrian representatives 
  \[
  \begin{cases}
    L_1 \,& p = 2,  \\
    L_1,-L_1 \,& p \neq 2\, \text{ and }\, q \equiv \pm1 \modp p, \\
    L_1,-L_1,L_2,-L_2  & \text{otherwise}.
  \end{cases}
  \]
  with 
  \[
    \tb_{\mathbb{Q}}(\pm L_1)= -\frac{p-q}{p} \;\; \text{and}\;\; \tb_{\mathbb{Q}}(\pm L_2) = -\frac{p-p'}{p}\\
  \]
  where $p'/q'$ is the largest rational number satisfying $pq' - p'q = -1$. Also the rational rotation numbers are determined by the formula in Lemma~\ref{lem:rotq-surgery} or~\ref{lem:rotq-Farey}. Every Legendrian representative of rational unknots in $\xi$ is Legendrian isotopic to one of the Legendrian representatives above, or their stabilization.
\end{theorem}

See Figure~\ref{fig:RP3} for the Legendrian mountain range of the rational unknot in $L(2,1)$. We also give the classification of (positive) transverse rational unknots in any tight contact structure on lens spaces up to transverse isotopy. 

\begin{figure}[htbp]{\scriptsize 
  \vspace{0.7cm}
    \begin{overpic}[tics=20]{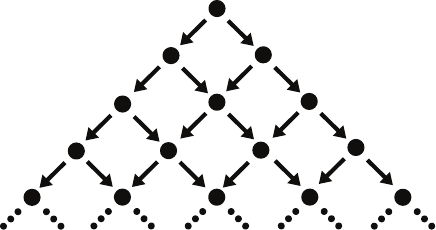}
    \put(74,125){$-1$}
    \put(102,125){$0$}
    \put(126,125){$1$}

    \put(-40, 103){$-1/2$}
    \put(-40, 81){$-3/2$}
    \put(-40, 58){$-5/2$}
    \put(-40, 36){$-7/2$}
    \put(-40, 14){$-9/2$}
  \end{overpic}}
  \vspace{0.1cm}
  \caption{The Legendrian mountain range of the rational unknot in $(\mathbb{R}\mathbb{P}^3,\xi_{std})$. Each dot represents a unique Legendrian representative with $(\rot_{\mathbb{Q}},\tb_{\mathbb{Q}})$.} 
  \label{fig:RP3}
\end{figure}
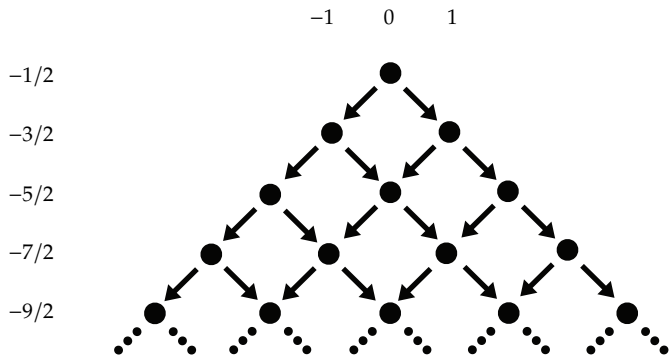

\begin{theorem} \label{thm:unknot-transverse}
  Suppose $p > q > 0$ and $\xi$ is a tight contact structure on $L(p,q)$. Rational unknots in $\xi$ are transversely simple: there are transverse representatives 
  \[
    \begin{cases}
      T_1 \,& p = 2,  \\
      T_1,\overline{T}_1 \,& p \neq 2\, \text{ and }\, q \equiv \pm1 \modp p, \\
      T_1,\overline{T}_1,T_2,\overline{T}_2  & \text{otherwise,}
    \end{cases}
  \]
  such that every transverse representative of rational unknots in $\xi$ is transversely isotopic to one of the transverse representatives above, or their stabilization. Also. $T_i$ is a positive transverse push-off of $L_i$ and $\overline{T}_i$ is a positive transverse push-off of $-L_i$.
\end{theorem}

There is an application of Theorem~\ref{thm:unknot-Legendrian} and \ref{thm:unknot-transverse}. In the forthcoming paper with Baker, Etnyre, and Onaran \cite{BEMO:torus}, we classify Legendrian and transverse positive torus knots in any tight contact structure on lens spaces. Basically, we classify the knots coarsely. However, due to Theorem~\ref{thm:unknot-Legendrian} and \ref{thm:unknot-transverse}, we could improve the result up to Legendrian and transverse isotopy. Unfortunately, there are some subtleties for negative torus knots (\dfn{e.g.}~Legendrian large cables, see \cite{CEM:cables}), so we could only improve the result for the negative torus knots with sufficiently negative $\tb_{\mathbb{Q}}$ and $\self_{\mathbb{Q}}$.

%---------------------------------------------------
\subsection*{Acknowledgements}
%---------------------------------------------------
The author thanks Anthony Conway for helpful comments on the first draft of the paper and John Etnyre for a useful conversation. The author also appreciates the referee for helpful comments.

%%%%%%%%%%%%%%%%%%%%%%%%%%%%%%%%%%%%%%%%%%%%%%%%%%%%%%%%%%%%
\section{Background and preliminary results}
%%%%%%%%%%%%%%%%%%%%%%%%%%%%%%%%%%%%%%%%%%%%%%%%%%%%%%%%%%%%

In this section, we review and prove some useful results on contact topology and the mapping class group of lens spaces that will be used throughout the paper. We assume the reader is familiar with $3$-dimensional contact topology, in particular, Legendrian and transverse knots and the convex surface theory. See \cite{Etnyre:knots,Geiges:book,Honda:classification1} for more details.

%--------------------------------------------------------
\subsection{Convex surfaces and bypasses} \label{sec:convex}
%--------------------------------------------------------
First, we warn the reader that our convention is that slopes of curves on the boundary of a solid torus with a preferred longitude are given by $\frac{\rm{meridian}}{\rm{longitude}}$. This has led to some differences between how we cite results and how they were initially stated.  

We will use several properties of convex surfaces without explicitly mentioning them (for more details, see \cite{Giroux:criteria, Giroux:classification,Giroux:cmcg,Honda:classification1}): perturbing a compact surface to be convex, realizing a particular characteristic foliation for the given dividing set, and using the Legendrian realization principle. We also assume that the boundary of any convex surface $\Sigma$ is Legendrian, if non-empty. When $\bd \Sigma$ is connected, then $\bd \Sigma$ is null-homologous and $tb(\bd \Sigma)$ is well-defined. Kanda \cite{Kanda:tb} proved
\begin{equation}\label{eq:tb}
  tb(\bd \Sigma) = -\frac{1}{2}\left|\bd \Sigma \cap \Gamma_\Sigma\right|,
\end{equation}
where $\Gamma_\Sigma$ is the dividing set of $\Sigma$. Suppose $\Sigma$ is a properly embedded convex surface in a contact $3$--manifold with convex boundary. The Euler class of the contact structure evaluates to $\chi(\Sigma_+) - \chi(\Sigma_-)$ on $\Sigma$ where $\Sigma_\pm$ are the positive/negative regions of the convex surface. If $(M,\xi)$ is a contact manifold with boundary, then we use the relative Euler class of $\xi$ to a nowhere vanishing section of $\xi|_{\bd M}$.

We can modify a convex surface by attaching a \dfn{bypass}, introduced by Honda \cite{Honda:classification1}. Consider a convex overtwisted disk whose dividing set consists of a single contractible closed curve. Take a properly embedded arc $\gamma$ on the disk intersecting the dividing curve in two points. By applying the Legendrian realization principle, we can assume that $\gamma$ is a Legendrian arc, and cut the disk along $\gamma$; each half-disk is called a bypass. Now, suppose a bypass $D$ transversely intersects a convex surface $\Sigma$ such that $D \cap \Sigma = \gamma$. Let $\Gamma_\Sigma$ be the dividing set of $\Sigma$. Since the dividing set interleaves, $\gamma$ intersects $\Gamma_\Sigma$ in three points. We call the Legendrian arc $\gamma$ on $\Sigma$ the \dfn{attaching arc} of the bypass $D$ and say $D$ is a bypass for $\Sigma$. After edge-rounding, the convex boundary of a neighborhood of $D \cup \Sigma$ is a surface isotopic to $\Sigma$ but with its dividing set changed in a neighborhood of the attaching arc as shown in Figure~\ref{fig:bypass-attachment}. We call this process a \dfn{bypass attachment along $\gamma$}. Note that Figure~\ref{fig:bypass-attachment} is drawn for the case that the bypass $D$ is attached ``from the front'', that is, sitting above the page. If we attach a bypass ``from the back'' of $\Sigma$, the result will be the mirror image of Figure~\ref{fig:bypass-attachment}. 

\begin{figure}[htbp]
  \vspace{0.2cm}
  \begin{overpic}[tics=20]{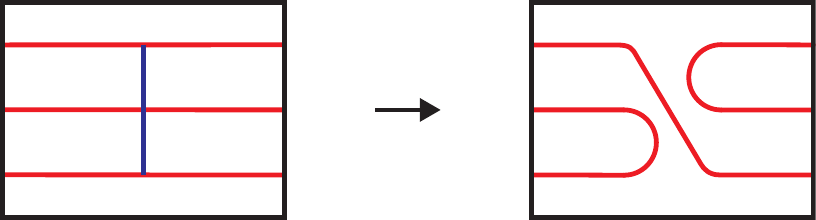}
  \end{overpic}
  \vspace{0.1cm}
  \caption{The effect of a bypass attachment from the front.}
  \label{fig:bypass-attachment}
\end{figure}

To study the effect of a bypass attachment on a torus, we first need to review the \dfn{Farey graph}. Given two rational numbers $a/b$ and $c/d$, we define their \dfn{Farey sum} to be 
\[
  \frac ab \oplus \frac cd = \frac{a+c}{b+d}.
\] 
We define their \dfn{Farey multiplication} to be
\[
  \frac ab \bigcdot \frac cd = ad - bc.
\]
We also define their \dfn{Farey subtraction} to be
\[
  \frac ab \ominus \frac cd = \frac{a-c}{b-d}.
\]
Take the Poincar\'e disk in $\R^2$ and label the points $(0,1)$ as $0=0/1$ and $(0,-1)$ as $\infty=1/0$. Take the half circle with non-negative $x$-coordinate. Pick a point in a half-way between two labeled points and label it with the Farey sum of the two points and connect it to both points by a geodesic. Repeat this process until all the positive rational numbers are a label on some point on the unit disk. Repeat the same for the half circle with non-positive $x$-coordinate (for $\infty$, use the fraction $-1/0$). We call this disk with the labels the Farey graph, see Figure~\ref{fig:Farey}. Also notice that two rational numbers $r$ and $s$ satisfy $|r \bigcdot s| = 1$ if and only if there is an edge between them in the Farey graph.

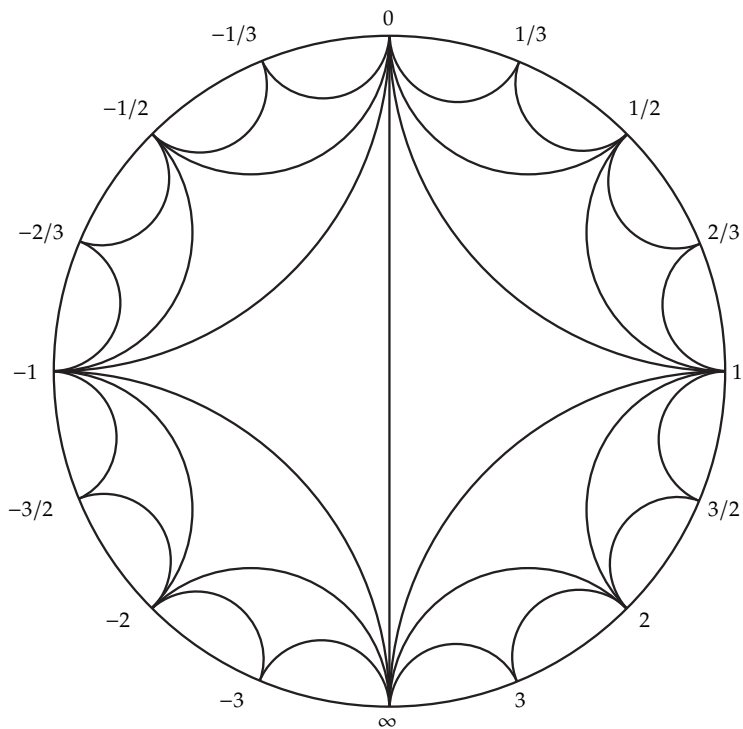
\begin{figure}[htbp]{\scriptsize
  \vspace{0.2cm}
  \begin{overpic}[tics=20]{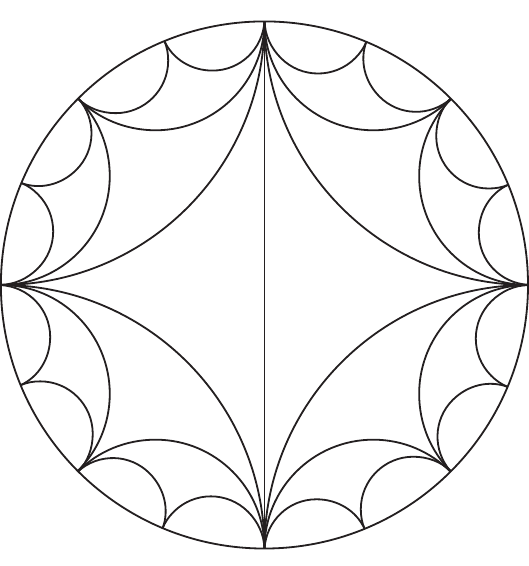}
    \put(123, 0){$\infty$}
    \put(125, 266){$0$}
    \put(-15, 132){$-1$}
    \put(257, 132){$1$}
    \put(20, 38){$-2$}
    \put(222, 38){$2$}
    \put(19, 232){$-1/2$}
    \put(218, 232){$1/2$}
    \put(60, 260){$-1/3$}
    \put(175, 260){$1/3$}
    \put(-13, 185){$-2/3$}
    \put(248, 185){$2/3$}
    \put(-17, 80){$-3/2$}
    \put(248, 80){$3/2$}
    \put(63, 8){$-3$}
    \put(175, 8){$3$}
  \end{overpic}}
  \vspace{0.1cm}
  \caption{The Farey graph.}
  \label{fig:Farey}
\end{figure}

We say a convex torus $T$ is a \dfn{standard convex torus} if the dividing set $\Gamma_T$ consists of two homologically essential closed curves and $T$ is foliated by Legendrian curves of any slope different from the dividing slope, called \dfn{Legendrian ruling curves}, and there are two singular lines parallel to the dividing curves, called \dfn{Legendrian divides}. By Giroux flexibility, we can make a $C^{0}$-small perturbation so that any convex torus becomes standard with any ruling slope (other than the dividing slope). Let $\gamma$ be an attaching arc of a bypass for $T$. Honda \cite{Honda:classification1} completely studied what happens when $\gamma$ is a part of a ruling curve for $T$. 

\begin{theorem}[Honda \cite{Honda:classification1}]\label{thm:bypass-torus}
	Suppose a standard convex torus $T$ has two dividing curves of slope $s$, and $\gamma$ is an attaching arc of a bypass for $T$, which is a part of a ruling curve of slope $r$. Let $T'$ be the convex torus obtained from $T$ by attaching a bypass along $\gamma$. Then the dividing set $\Gamma_{T'}$ consists of two dividing curves of slope $s'$, where 
  \begin{itemize}
    \item if the bypass is attached from the front, then $s'$ is the farthest point from $s$ on the Farey graph clockwise of $s$ and counterclockwise of $r$ that is connected to $s$ by an edge (and if $s$ and $r$ are connected by an edge, then $s' = r$),
    \item if the bypass is attached from the back, then $s'$ is the farthest point from $s$ on the Farey graph counterclockwise of $s$ and clockwise of $r$ that is connected to $s$ by an edge (and if $s$ and $r$ are connected by an edge, then $s' = r$).
  \end{itemize}
\end{theorem}

In general, we can find a bypass lying on a convex surface if there exist at least two dividing curves and a properly embedded boundary-parallel dividing curve on the surface by applying the Legendrian realization principle. In particular, if $\chi(\Sigma) < 1$ and there exists only one dividing curve on the surface which is properly embedded and boundary-parallel, then we can wiggle the surface and increase the number of dividing curves so we can still find a bypass. 

\begin{theorem}[Honda \cite{Honda:classification1}]\label{thm:imba} Let $\Sigma$ be a convex surface and $D$ be a convex disk with Legendrian boundary. Suppose $\Sigma$ and $D$ intersect transversely and $\Sigma \cap D = \bd D$. Suppose $tb(\bd D) < -1$. Then for any boundary-parallel dividing curve $d$ on $D$, there exists a bypass for $\Sigma$ containing $d$.
\end{theorem}

%------------------------------------------------------------------
\subsection{Bypasses and contact isotopy} \label{sec:bypass}
%------------------------------------------------------------------
We continue to review the properties of bypasses. Let $\Sigma$ be a convex surface and $D$ be a bypass for $\Sigma$. Suppose the attaching arc of $D$ passes three dividing curves $d_1$, $d_2$ and $d_3$ consecutively. We say the bypass $D$ is \dfn{effective} if $d_2$ is different from $d_1$ and $d_3$. Honda showed \cite{Honda:classification1} attaching an effective bypass to a torus will decrease the number of dividing curves if $d_1$, $d_2$ and $d_3$ are all different, or change the dividing slope of $T$ if $d_1$ and $d_3$ are the same (Theorem~\ref{thm:bypass-torus}). 

Suppose $D$ is a non-effective bypass for a convex surface $\Sigma$ and let $\Sigma'$ be the resulting convex surface after attaching the bypass $D$ to $\Sigma$. Define $|\Gamma_{\Sigma}|$ to be the number of dividing curves on $\Sigma$. There are three types of non-effective bypasses for $\Sigma$ according to the effect on the dividing set:

\begin{enumerate}
  \item $\Gamma_{\Sigma'} = \Gamma_\Sigma$,
  \item $\Gamma_{\Sigma'}$ contains a contractible closed curve and $|\Gamma_{\Sigma'}| > |\Gamma_{\Sigma}|$,
  \item $\Gamma_{\Sigma'} \neq \Gamma_\Sigma$ and $|\Gamma_{\Sigma'}| \geq |\Gamma_{\Sigma}|$.
\end{enumerate} 

See Figure~\ref{fig:non-effective} for the first two cases. Recall that Giroux \cite{Giroux:criteria} proved that an $I$-invariant neighborhood of a convex surface $\Sigma$ is tight if and only if $\Sigma \not\cong S^2$ and there is no closed contractible dividing curve on $\Sigma$, or $\Sigma \cong S^2$ and there is a single dividing curve on $\Sigma$. Thus the second type of bypasses does not occur in a tight contact structure. If a bypass does not change the dividing set, we call it a \dfn{trivial bypass}. Honda \cite{Honda:bypass} showed that a trivial bypass is indeed trivial. 

\begin{lemma}[Honda \cite{Honda:bypass}]\label{lem:trivial}
  Suppose $\Sigma$ is a convex surface which is closed or compact with Legendrian boundary. If $D$ is a trivial bypass for $\Sigma$, then a neighborhood $N(\Sigma \cup D)$, which is a result of the bypass attachment, is an $I$-invariant neighborhood of $\Sigma$. 
\end{lemma}

\begin{figure}[htbp]
  \vspace{0.2cm}
  \begin{overpic}[tics=20]{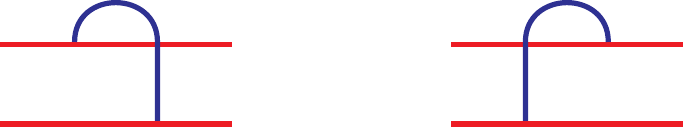}
  \end{overpic}
  \vspace{0.1cm}
  \caption{Two attaching arcs of non-effective bypasses.}
  \label{fig:non-effective}
\end{figure}

A \dfn{rotative layer} is a tight contact structure on $T^2 \times I$ with convex boundary such that the dividing slopes of $T^2 \times \{0\}$ and $T^2 \times \{1\}$ are different. Also, a \dfn{non-rotative layer} is a tight contact structure on $T^2 \times [0,1]$ with convex boundary such that any convex tori parallel to the boundary have the same dividing slope. Non-rotative layers were studied in \cite{Honda:non-rotative, HKM:attach=dig}. One useful property is the attach=dig principle. We introduce a version of the principle for a simple case.

\begin{theorem}[The attach=dig principle, Honda--Kazez--Mati\'c \cite{HKM:attach=dig}]\label{thm:attach=dig}
  Let $(T^2 \times [0,4], \xi)$ be a rotative layer. Denote $T^2 \times \{i\}$ by $T_i$ for $i \in \mathbb{Z}$ and suppose they are convex. Let $s_i$ and $n_i$ be the dividing slope and the number of dividing curves on $T_i$, respectively. Suppose $s_0 < s_2 < s_4$. Then after contact isotopy relative to $T_2$ and the boundary, $T^2 \times [1,3]$ becomes an $I$-invariant neighborhood with $s_1 = s_2 = s_3$ and $n_1 = n_3 = 2$. Also, $T_1$ can be obtained by attaching a sequence of bypasses from the back of $T_2$, and $T_3$ can be obtained by attaching a sequence of bypasses from the front of $T_2$.
\end{theorem}

Let $\Sigma$ be a closed surface or a compact surface with boundary. Giroux \cite{Giroux:cmcg} showed that we can perturb a contact structure on $\Sigma \times [0,1]$ so that $\Sigma \times \{t\}$ are convex for all but finite $t \in [0,1]$, and a neighborhood of the non-convex $\Sigma \times \{t\}$ is contactomorphic to a bypass attachment. After that, Honda and Huang \cite{HH:parametric} generalized it to every dimension.

\begin{theorem}[Giroux \cite{Giroux:cmcg}, Honda--Huang \cite{HH:parametric}]\label{thm:parametric}
  Let $\xi$ be a contact structure on $\Sigma \times [0,1]$ such that $\Sigma \times \{0\}$ and $\Sigma \times \{1\}$ are convex. Then up to contact isotopy relative to the boundary, there exists a finite sequence $0 < t_1 < \cdots < t_n < 1$ such that  
  \begin{itemize}
    \item $\Sigma \times \{t\}$ is convex except for $t = t_i$.
    \item There exists $\epsilon > 0$ for each $i$ such that $\Sigma \times [t_i - \epsilon, t_i + \epsilon]$ is contactomorphic to a bypass attachment.
  \end{itemize}
\end{theorem}

Colin \cite{Colin:discretization} improved this result for a one-parameter family of embedded surfaces in a contact $3$--manifold.  

\begin{theorem}[Isotopy discretization, Colin \cite{Colin:discretization}, see also Honda \cite{Honda:bypass}]\label{thm:discretization}
  Let $(M,\xi)$ be a contact $3$--manifold and $\Sigma, \Sigma'$ be embedded convex surfaces in $M$ which are closed or compact with identical Legendrian boundary. Suppose $\Sigma$ and $\Sigma'$ are smoothly isotopic relative to boundary. Then there exists a finite sequence of convex surfaces $(\Sigma_1, \Sigma_2, \cdots, \Sigma_n)$ such that $\Sigma_1 = \Sigma$, $\Sigma_n = \Sigma'$ and $\Sigma_{i+1}$ is obtained by a single bypass attachment to $\Sigma_i$.  
\end{theorem}

We end this section by showing that if a contactomorphism fixes a convex surface or a Legendrian knot, then after contact isotopy, the contactomorphism also fixes a neighborhood of them. The second statement of Lemma~\ref{lem:fix} was proved in \cite[Lemma~6]{DG:S1xS2}, but we present a proof for completeness.

\begin{lemma}\label{lem:fix}
  Let $C$ be a subset in a compact contact $3$--manifold $(M,\xi)$ and $f\colon (M,\xi) \to (M,\xi)$ be a contactomorphism. Suppose $f|_C = id$. Then,  
  \begin{enumerate}
    \item if $C = \Sigma$ is a compact convex surface, then there exist an $I$-invariant neighborhood $N$ of $\Sigma$ and a contactomorphism $\widetilde{f}$ of $(M,\xi)$ such that $\widetilde{f}|_N = id$ and $\widetilde{f}$ is contact isotopic to $f$.  
    \item if $C = L$ is a Legendrian knot, then there exist a standard neighborhood $N$ of $L$ and a contactomorphism $\widetilde{f}$ of $(M,\xi)$ such that $\widetilde{f}|_N = id$ and $\widetilde{f}$ is contact isotopic to $f$.
  \end{enumerate}
\end{lemma}

\begin{proof}
  Consider the case $C = \Sigma$ first. Take an $I$-invariant neighborhood of $\Sigma$ which is strictly contactomorphic to $(\Sigma \times \mathbb{R},\, \beta + u\,dt)$ where $\Sigma = \Sigma \times \{0\}$, $\beta \in \Omega^1(\Sigma)$ and $u\colon \Sigma \to \mathbb{R}$. Take another small neighborhood $N = \Sigma \times [-\epsilon, \epsilon]$ satisfying $f(N) \subset \Sigma \times \mathbb{R}$.  
  We will use the following strategy: we will find an isotopy of contact embeddings $f_s\colon (N,\xi|_N) \to (M,\xi)$ where $f_0 = f|_N$ and $f_1 = id$. According to the contact isotopy extension theorem~\cite[Theorem~2.6.12]{Geiges:book}, there exists a contact isotopy $\phi_s\colon (M,\xi) \to (M,\xi)$ satisfying $\phi_0 = id$ and $\phi_s \circ f_0 = f_s$. Then $\widetilde{f} := \phi_1 \circ f$ is our desired contactomorphism.

  Let $v_0 := \partial_t$ and $v_1 := f_*(\partial_t)$. It is not hard to check $\mathcal{L}_{v_0}\alpha = 0$ and $\mathcal{L}_{v_1}\alpha = \lambda\,\alpha$ where $\alpha = \beta + u\,dt$ and $\lambda\colon \Sigma \times \mathbb{R} \to \mathbb{R}$, so both $v_0$ and $v_1$ are contact vector fields transverse to $\Sigma$. Notice that $v_1$ is well-defined on $f(N)$ but we can extend it to entire $\Sigma \times \mathbb{R}$ by extending the corresponding contact Hamiltonian. Now we define $v_s := sv_0 + (1-s)v_1$. Then for every $s \in [0,1]$, the vector field $v_s$ is also a contact vector field since 
  \begin{align*}
    \mathcal{L}_{v_s}\alpha &= s\mathcal{L}_{v_0}\alpha + (1-s)\mathcal{L}_{v_1}\alpha\\
    &= (1-s)\lambda\,\alpha.
  \end{align*} 
  Let $\psi^t_s$ be the flow of $v_s$. Since $v_0 = \partial_t$, we have 
  \[
    \psi^t_0(p,0) = (p,t) \,\text{ for }\, (p,t) \in \Sigma \times \mathbb{R}.
  \] 
  Also, since $v_1 = f_*(v_0)$ on $N$ and $f|_{\Sigma} = id$, we have 
  \begin{align*}
    \psi^t_1(p,0) &= f \circ \psi^t_0 \circ f^{-1}(p,0)\\
    &= f \circ \psi^t_0(p,0)\\    % naturality of flows. Corollary 9.14 of Lee, smooth manifolds
    &= f(p,t)
  \end{align*}
  for $(p,t) \in N$. Define an isotopy of contact embeddings $f_s(p,t) := \psi^t_{1-s}(p,0)$ for $(p,t) \in N$ and it is our desired isotopy.

  Now consider the case $C = L$. Take a standard neighborhood of $L$ which is strictly contactomorphic to $(S^1 \times \mathbb{R}^2, dz - y\,dx)$ where $x \sim x+1$ is the coordinate on $S^1 = \mathbb{R}/\mathbb{Z}$, the pair $(y,z)$ is the coordinates on $\mathbb{R}^2$ and $L$ is identified with $S^1 \times \{0\}$. Take another standard neighborhood $N = S^1 \times D^2$ where $D^2$ is a small disk containing the origin such that $f(N) \subset S^1 \times \mathbb{R}^2$. We will follow the same strategy as in the case of $C = \Sigma$. That is, it is enough to find an isotopy of contact embeddings $f_s\colon (N,\xi|_N) \to (M,\xi)$ satisfying $f_0 = f|_N$ and $f_1 = id$. 

  We can write $f|_N$ in the form 
  \[
    f|_N(x,y,z) = (u(x,y,z),v(x,y,z),w(x,y,z))
  \] 
  where $u\colon N \to S^1$, $v\colon N \to \mathbb{R}$ and $w\colon N \to \mathbb{R}$. In the local coordinates, we can rewrite the condition $f^*(\alpha) = \lambda\alpha$ where $\lambda\colon N \to \mathbb{R}^+$ for $f|_N$ to be a contact embedding as follows: 
  \[
    dw - v\,du = \lambda(dz - y\,dx),
  \]
  which is equivalent to 
  \begin{equation}\label{eq:contact}
  \left\{\begin{array}{rcl}  
    \displaystyle \frac{\partial w}{\partial x} - v\frac{\partial u}{\partial x} & = & -\lambda y,\\ 
    \displaystyle \frac{\partial w}{\partial y} - v\frac{\partial u}{\partial y} & = & 0,\rule{0cm}{9mm}\\
    \displaystyle \frac{\partial w}{\partial z} - v\frac{\partial u}{\partial z} & = & \lambda.\rule{0cm}{9mm}
  \end{array}\right. 
  \end{equation} 
  Since $f|_L = id$, we have
  \begin{equation}\label{eq:fix}
    u(x,0,0) = x,\quad v(x,0,0) = w(x,0,0) = 0.
  \end{equation} 
  Notice that for $s>0$, the dilation $\delta_s(x,y,z) = (x, sy, sz)$ is a strict contactomorphism of $(S^1 \times \mathbb{R}^2, dz - y\,dx)$. Thus we have an isotopy of contact embeddings
  \[
    g_s := \delta_s^{-1} \circ f|_N \circ \delta_s (x,y,z) = (u(x,sy,sz), \frac{1}{s} v(x,sy,sz), \frac{1}{s} w(x,sy,sz)).
  \]
  Let $\lambda_0(x) := \lambda(x,0,0)$. Since $u$, $v$ and $w$ are $C^\infty$, we have
  \[
    g_0 := \lim_{s\to0}g_s = (x,\, y \cdot v_y(x,0,0) + z \cdot v_z(x,0,0),\, y \cdot w_y(x,0,0) + z \cdot w_z(x,0,0)).
  \] 
  Differentiate the first equation in (\ref{eq:contact}) with respect to $z$ and we obtain
  \[
    w_{xz} - v_z u_x - v u_{xz} = -y\lambda_z. 
  \]
  Evaluate this equation at $(x,0,0)$. Then by the equations in~(\ref{eq:fix}), we obtain 
  \begin{equation}\label{eq:1}
    v_z(x,0,0) = w_{xz}(x,0,0).
  \end{equation}
  Differentiate the third equation in (\ref{eq:contact}) with respect to $x$ and we obtain
  \[
    w_{xz} - v_x u_z - v u_{xz} = \lambda_x. 
  \]
  Evaluate this equation at $(x,0,0)$. Then by the equations in~(\ref{eq:fix}), we obtain 
  \begin{equation}\label{eq:2}
    w_{xz}(x,0,0) = \lambda_x(x,0,0) = \lambda'_0(x).
  \end{equation} 
  Differentiate the second equation in (\ref{eq:contact}) with respect to $x$ and we obtain
  \[
    w_{xy} - v_x u_y - v u_{xy} = 0. 
  \]
  Evaluate this equation at $(x,0,0)$. Then by the equations in~(\ref{eq:fix}), we obtain
  \begin{equation}\label{eq:wxy}
    w_{xy}(x,0,0)=0.
  \end{equation}
  Differentiate the first equation in (\ref{eq:contact}) with respect to $y$ and we obtain
  \[
    w_{xy} - v_y u_x - v u_{xy} = -\lambda - y\lambda_y. 
  \]
  Evaluate this equation at $(x,0,0)$. Then by the equations in~(\ref{eq:fix})  and (\ref{eq:wxy}), we obtain 
  \begin{equation}\label{eq:3}
    v_y(x,0,0) = \lambda(x,0,0) = \lambda_0(x).
  \end{equation}  
  Evaluate the equations in~(\ref{eq:contact}) at $(x,0,0)$. Then by the equations in~(\ref{eq:fix}), we obtain
  \begin{equation}\label{eq:4}
    w_y(x,0,0) = 0, \quad w_z(x,0,0) = \lambda_0(x). 
  \end{equation}
  Finally, from the equations~(\ref{eq:1}),~(\ref{eq:2}),~(\ref{eq:3}) and~(\ref{eq:4}), we obtain 
  \[
    g_0(x,y,z) = (x,\, y \cdot \lambda_0(x) + z \cdot \lambda_0'(x),\, z \cdot \lambda_0(x)),
  \]
  which is a contact embedding from $(N,\xi|_N)$ to $(S^1 \times \mathbb{R}^2, dz - y\,dx)$. Now define $\lambda_s(x) := s + (1-s)\lambda_0(x)$. Then we can define another isotopy of contact embeddings as follows:  
  \[
    h_s(x,y,z) := (x,\, y \cdot \lambda_s(x) + z \cdot \lambda_s'(x),\, z \cdot \lambda_s(x)).
  \]
  Let $f_s$ be a concatenation of $g_{1-s}$ and $h_s$. This is our desired isotopy.
\end{proof}

%-------------------------------------------------------------------------------------------------------------
\subsection{Tight contact structures on a solid torus and lens spaces} \label{sec:lens}
%-------------------------------------------------------------------------------------------------------------
Consider a tight contact structure $\xi$ on $T(s_1,s_2) = T^2 \times I$ with a characteristic foliation $\mathcal F$ on the boundary that is divided by two dividing curves of slope $s_i$ on $T^2 \times \{i\}$ for $i=0,1$, where $s_0$ and $s_1$ are connected by an edge in the Farey graph. We say that a contact structure $\xi$ is \dfn{minimally twisting} if for any boundary-parallel convex torus $T$ in $\xi$, the dividing slope is clockwise of $s_0$ and counterclockwise of $s_1$ in the Farey graph.

\begin{theorem}[Giroux \cite{Giroux:classification}, Honda \cite{Honda:classification1}]\label{thm:basic-slice}
  If $T(s_1,s_2)$ and $\mathcal F$ are as above, then there exist exactly two minimally twisting tight contact structures on $T(s_1,s_2)$ that induce $\mathcal F$ on the boundary, up to isotopy fixing $\mathcal F$.
\end{theorem}

The two contact structures given by Theorem~\ref{thm:basic-slice} are distinguished by their relative Euler class. Let $A = S^1 \times I \subset T^2 \times I$ be an oriented annulus where the slope of $S^1 \times \{0,1\}$ is $s_2$. We call a basic slice \dfn{positive} or \dfn{negative basic slices}, respectively, according to the sign of the relative Euler class evaluated on $A$.

Let $V = S^1 \times D^2$ and choose coordinates for $H_1(\bd V)$ such that $0$ is a longitude $S^1 \times \{p\}$ (product framing), and $\infty$ is a meridian. Let $p/q$ is a rational number and $k$ be the unique integer such that $\frac{p+kq}{q} \in [-1,0)$, and  
\[
  \frac q{p+kq} = [r_0, \ldots, r_n] = r_0-\frac{1}{r_1-\frac{1}{\cdots - \frac{1}{r_n}}}
\]
where $r_n \leq -1$, and $r_i \leq -2$ for $i=0,\ldots,n-1$.
  
\begin{theorem}[Giroux \cite{Giroux:classification}, Honda \cite{Honda:classification1}]\label{thm:classification-s1d2}
  Suppose $V = S^1 \times D^2$ with two dividing curves of slope $s = p/q$. Fix a characteristic foliation $\mathcal F$ on $\partial V$ that is divided by the dividing curves. Then  
  \begin{enumerate}
    \item there are $\left|(r_0+1)\cdots(r_{n-1}+1)r_n\right|$ tight contact structures up to isotopy fixing $\mathcal{F}$,
    \item there is one to one correspondence between tight contact structures on $V$ and $T(\lfloor s \rfloor,s)$,
    \item if $s \in \mathbb{Z}$, there is a unique tight contact structure and it is universally tight, 
    \item if $s \notin \mathbb{Z}$, there are exactly two universally tight contact structures,
    \item a tight contact structure on $V$ is universally tight if and only if it has the extremal relative Euler class, which evaluates to $\pm(|q| - 1)$ on a convex meridian disk $D$ of $V$ whose boundary intersects the dividing curves on $\partial V$ minimally. In this case, all dividing curves on $D$ is boundary-parallel. See Figure~\ref{fig:univ} for example.
  \end{enumerate} 
\end{theorem}

To study tight contact structures on lens spaces, it is useful to use different coordinates for a solid torus. First, we fix a homology basis of $T^2 \times [0,1]$ so that the slopes of curves on $T^2 \times \{t\}$ are well-defined. We say a solid torus is a \dfn{solid torus with lower meridian of slope r} if it is formed from $T^2\times [0,1]$ by collapsing the leaves of a linear foliation (closed curves) on $T^2\times \{0\}$ of slope $r$. We denote it by $S_r$. If there is a contact structure on $S_r$ such that the boundary is convex with two dividing curves of slope $s$, then we denote the contact manifold by $S_r(s)$. We also note that any convex torus in $S_r(s)$ parallel to the boundary has a dividing slope clockwise of $r$ and counterclockwise of $s$ in the Farey graph. We can similarly define the \dfn{solid torus with upper meridian of slope $r$} similarly except we collapse leaves of a linear foliation on $T^2\times \{1\}$ of slope $r$ and denote the result by $S^r$. We also denote by $S^r(s)$ the contact structure on $S^r$ with convex boundary having two dividing curves of slope $s$. We note that any convex torus in $S^r(s)$ parallel to the boundary has a dividing slope counterclockwise of $r$ and clockwise of $s$ in the Farey graph.

According to Theorem~\ref{thm:classification-s1d2}, both $S_r(s)$ and $S^r(s)$ admit a unique tight contact structure if and only if there is an edge between $s$ and $r$ in the Farey graph. We assume a solid torus $S^1 \times D^2$ has lower meridian of slope $\infty$ unless otherwise specified.

Recall the standard contact structure $\xi_{std}$ on $S^3$ is a union of standard neighborhoods of the Legendrian Hopf link $L_1 \cup L_2$ with $\tb(L_1) = \tb(L_2) = -1$. This gives a decomposition of $\xi_{std}$ into $S^0(-1)$ and $S_{\infty}(-1)$, where $L_1$ and $L_2$ are the cores of $S^0(-1)$ and $S_{\infty}(-1)$, respectively. Suppose $(p,q)$ is a pair of coprime integers satisfying $p > q > 0$. According to Giroux \cite{Giroux:classification} and Honda \cite{Honda:classification1}, we can obtain any tight contact structure on a lens space $L(p,q)$ by performing contact $-(p/q-1)$--surgery on $L_2$, that is, remove a standard neighborhood $S_{\infty}(-1)$ of $L_2$ and glue $S_{-p/q}(-1)$ to the complement. Thus any tight contact structure on $L(p,q)$ can be decomposed into $S^0(-1) \cup S_{-p/q}(-1)$. See the first drawing of Figure~\ref{fig:unknots}. Notice that $L_1$ still has a standard neighborhood $S^0(-1)$. 

We can also represent this decomposition on the Farey graph. Let $s_0,\ldots,s_n$ be the shortest path from $-p/q$ to $0$ in the Farey graph clockwise of $-p/q$ and counterclockwise of $0$. Notice that $s_0=-p/q$, $s_{n-1}=-1$ and $s_n = 0$. Decorate the edges in the path with $+$ or $-$ except for the first and the last ones. Each decorated edge from $s_i$ to $s_{i+1}$ represents a basic slice $T(s_i,s_{i+1})$ and the decoration on the edge represents the sign of the basic slice. The first and the last edges represent the solid tori $S_{-p/q}(s_1)$ and $S^0(-1)$, respectively. We can consider $S_{-p/q}(-1)$ as a union of $S_{-p/q}(s_1)$ and $T(s_1,-1)$. See Figure~\ref{fig:path}. For later usage, notice that $s_1 = (p'-p)/(q-q')$ where $p'/q'$ is the largest (extended) rational number satisfying $pq' - p'q = -1$. We set $(p',q')=(1,0)$ if $q \equiv 1 \pmod p$.

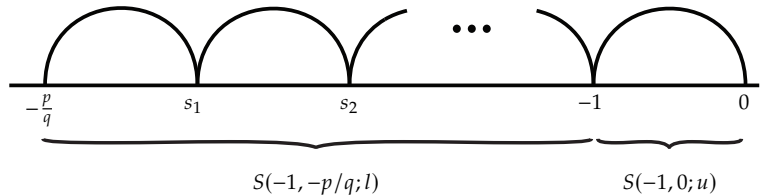
\begin{figure}[htbp]{\scriptsize
  \vspace{0.1cm}
  \begin{overpic}[tics=20]{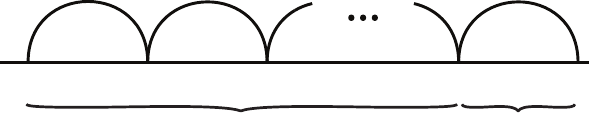}
    \put(6,13){$-\frac pq$}
    \put(66,15){$s_1$}
    \put(125,15){$s_2$}
    \put(215,15){$-1$}
    \put(276,15){$0$}

    \put(92,-15){$S_{-p/q}(-1)$}
    \put(232,-15){$S^0(-1)$}
  \end{overpic}}
  \vspace{0.4cm}
  \caption{A path in the Farey graph representing a tight contact structure on $L(p,q)$.}
  \label{fig:path}
\end{figure}

Let $(p,q)$ be a pair of coprime integers satisfying $p > q > 0$ and 
\[
  -\frac pq = [r_0, \ldots, r_n] 
\]
where $r_i \leq -2$ for $0 \leq i \leq n$.

\begin{theorem}[Giroux \cite{Giroux:classification}, Honda \cite{Honda:classification1}]\label{thm:classification-lens}
  Suppose $(p,q)$ are as above. Then    
  \begin{enumerate}
    \item there are $\left|(r_0+1)\cdots(r_{n-1}+1)(r_n+1)\right|$ tight contact structures on $L(p,q)$ up to isotopy,
    \item if $q \equiv -1 \pmod p$, there exists a unique tight contact structure on $L(p,q)$ and it is universally tight, 
    \item if $q \not\equiv -1 \pmod p$, there are exactly two universally tight contact structures on $L(p,q)$,
    \item Any tight contact structure $\xi$ on $L(p,q)$ can be decomposed into tight contact structures on $S_{-p/q}(-1)$ and $S^0(-1)$. In particular, $\xi$ is universally tight if and only if the contact structure $\xi$ restricted to $S_{-p/q}(-1)$ is universally tight. 
  \end{enumerate}
\end{theorem} 

There is another way to construct universally tight contact structures on lens spaces $L(p,q)$. Consider $S^3$ as a unit sphere in $\mathbb{C}^2$. Then the standard contact structure $\xi_{std}$ on $S^3$ is the kernel of 
\[
  \alpha = (x_1dy_1 - y_1dx_1 + x_2dy_2 - y_2dx_2)|_{S^3}.
\]
We can consider $L(p,q)$ as the quotient of $S^3 \subset \mathbb{C}^2$ under the $\mathbb{Z}_p$-action generated by
\[
  (z_1,z_2) \mapsto (e^{2\pi i/p}z_1, e^{2\pi qi / p}z_2).
\]
Since $\alpha$ is invariant under this $\mathbb{Z}_p$-action, we obtain an induced contact structure on $L(p,q)$. We call this contact structure the \dfn{standard contact structure $\xi_{std}$} on $L(p,q)$. However, one should notice that the standard contact structure is not unique in general. In fact, we can repeat the same construction on $-\alpha$ and obtain another contact structure on $L(p,q)$. Although $\xi_{std}$ and $-\xi_{std}$ are isotopic in $S^3$, this is not the case for the induced contact structures on lens spaces in general. Thus we denote them by $\xi_{std}^+$ and $\xi_{std}^-$, respectively. According to Theorem~\ref{thm:classification-lens}, $\xi_{std}^+$ and $\xi_{std}^-$ are isotopic if and only if $q \equiv -1 \pmod p$. Thus in this case, we just denote them by $\xi_{std}$. Even in the case of $q \not\equiv -1 \pmod p$, since most of the arguments work in the same way, we will frequently denote them by $\xi_{std}$ and this means that we fix one of two standard contact structures.  

If $q^2 \equiv 1 \pmod p$, there exists an orientation preserving diffeomorphism $\sigma$ on $L(p,q)$, which is defined by
\begin{align*}
  \sigma \colon L(p,q) &\to L(p,q)\\
  (z_1,z_2) &\mapsto (z_2, z_1)
\end{align*}
Also, there exists an orientation preserving diffeomorphism $\tau$ on any lens spaces $L(p,q)$, which is defined by
\begin{align*}
  \tau \colon L(p,q) &\to L(p,q)\\
  (z_1,z_2) &\mapsto (\overline{z}_1, \overline{z}_2)
\end{align*}
We can check $\sigma$ is a coorientation preserving contactomorphism of $\xi_{std}$ as follows:   
\begin{align*}
  \sigma^*(\alpha) &= \sigma^*(x_1dy_1 - y_1dx_1 + x_2dy_2 - y_2dx_2)\\
  &= x_2dy_2 - y_2dx_2 + x_1dy_1 - y_1dx_1\\
  &= \alpha.
\end{align*}
We can also check $\tau$ is a coorientation reversing contactomorphism of $\xi_{std}$ as follows:
\begin{align*}
  \tau^*(\alpha) &= \tau^*(x_1dy_1 - y_1dx_1 + x_2dy_2 - y_2dx_2)\\
  &= -x_1dy_1 + y_1dx_1 - x_2dy_2 + y_2dx_2\\
  &= -\alpha.
\end{align*}
If $q \equiv -1 \pmod p$, since $\xi_{std}^+$ and $\xi_{std}^-$ are isotopic, we can apply the Moser's trick (see \cite[Theorem~2.2.2]{Geiges:book}) and find an isotopy $\psi_t$ such that $(\psi_1)_*(\xi_{std}^\pm) = \xi_{std}^\mp$. Thus $\psi_1 \circ \tau$ is a coorientation preserving contactomorphism of $\xi_{std}$ which is smoothly isotopic to $\tau$. We denote this by $\overline{\tau}$.

%---------------------------------------------------------------------------------------------------
\subsection{The mapping class group and rational unknots in lens spaces} \label{sec:mcg-lens}
%---------------------------------------------------------------------------------------------------
The mapping class group of lens spaces was determined by Bonahon \cite{Bonahon:lens}. We warn the reader that his definition of lens spaces is different from ours. He defined $L(p,q)$ to be $p/q$-surgery on the unknot in $S^3$, while we defined $L(p,q)$ to be $-p/q$-surgery on the unknot in $S^3$ (which is commonly used by contact topologists). Thus some statements below are different from the ones initially stated. 

Recall that in Section~\ref{sec:lens}, we defined two diffeomorphisms, $\sigma$ by $(z_1,z_2) \mapsto (z_2,z_1)$ if $q^2 \equiv 1 \pmod p$, and $\tau$ by $(z_1,z_2)\mapsto(\overline{z}_1,\overline{z}_2)$. Bonahon \cite{Bonahon:lens} determined the mapping class group of lens spaces in terms of $\sigma$ and $\tau$. 

\begin{theorem}[Bonahon \cite{Bonahon:lens}]\label{thm:mcg-lens}
  The mapping class group of $L(p,q)$ is 
  \[
    \pi_0(\Diff_+(L(p,q))) = \begin{cases} 
      \mathbb{Z}_2 \oplus \mathbb{Z}_2 \cong \langle \sigma, \tau \rangle \,& p \neq 2,\; q \not\equiv \pm1 \modp p\, \text{ and }\, q^2 \equiv 1 \modp p,\\
      \mathbb{Z}_2 \cong \langle \sigma \rangle \cong  \langle \tau \rangle \,& p \neq 2\, \text{ and }\, q \equiv -1 \modp p,\\
      \mathbb{Z}_2 \cong \langle \tau \rangle \,& p \neq 2\, \text{ and }\, q \equiv 1 \modp p,\\
      \mathbb{Z}_2 \cong \langle \tau \rangle \,& p \neq 2\, \text{ and }\, q^2 \not\equiv 1 \modp p,\\
      1 & p = 2. \end{cases}
  \]
\end{theorem}

It will be useful to consider the mapping class group of lens spaces relative to a Heegaard torus. Let $T$ be a Heegaard torus of $L(p,q)$, which is unique up to smooth isotopy by Bonahon \cite{Bonahon:lens}. Define $\Diff_+(L(p,q);T)$ to be the group of orientation preserving diffeomorphisms fixing $T$ setwise. Bonahon \cite{Bonahon:lens} determined the mapping class group of $L(p,q)$ relative to $T$.

\begin{theorem}[Bonahon \cite{Bonahon:lens}]\label{thm:mcg-lens-torus}
  The mapping class group of $L(p,q)$ relative to $T$ is 
  \[
    \pi_0(\Diff_+(L(p,q);T)) = \begin{cases} 
      \mathbb{Z}_2 \oplus \mathbb{Z}_2 \cong \langle \sigma, \tau \rangle \,& q^2 \equiv 1 \modp p,\\
      \mathbb{Z}_2 \cong \langle \tau \rangle \,& q^2 \not\equiv 1 \modp p.\end{cases}
  \]
\end{theorem}

Bonahon \cite{Bonahon:lens} also studied the natural inclusion $i:\Diff_+(L(p,q);T)) \hookrightarrow \Diff_+(L(p,q))$ at the $\pi_0$ level.

\begin{theorem}[Bonahon \cite{Bonahon:lens}]\label{thm:mcg-lens-kernel}
  The induced map from the natural inclusion
  \[
    i_*\colon \pi_0(\Diff_+(L(p,q);T)) \to \pi_0(\Diff_+(L(p,q))) 
  \]
  is surjective and the kernel is
  \[
    \ker i_* = \begin{cases} 
      \mathbb{Z}_2 \oplus \mathbb{Z}_2 \cong \langle \sigma, \tau \rangle \,& p = 2,\\
      \mathbb{Z}_2 \cong \langle \sigma\circ\tau \rangle \,& p \neq 2\, \text{ and }\, q \equiv -1 \modp p,\\
      \mathbb{Z}_2 \cong \langle \sigma \rangle \,& p \neq 2\, \text{ and }\, q \equiv 1 \modp p,\\
      1 \,& q \not\equiv \pm1 \modp p.\end{cases}
  \]
\end{theorem}

Now Theorem~\ref{thm:mcg-lens} easily follows from Theorem~\ref{thm:mcg-lens-torus} and~\ref{thm:mcg-lens-kernel}.

A knot $K$ in a $3$--manifold is a \dfn{rational unknot} if it is rationally null-homologous and its minimal rational Seifert genus is $0$. Baker and Etnyre \cite{BE:rational} showed that rational unknots in lens spaces are cores of the Heegaard torus $T$, which is unique up to smooth isotopy by Bonahon \cite{Bonahon:lens}. Since $T$ bounds two solid tori, we can define two oriented rational unknots $K_1$ and $K_2$ as follows:  
\begin{align*}
  K_1 &= \{(e^{i\theta},0) : 0 \leq \theta \leq \frac{2\pi}{p}\},\\
  K_2 &= \{(0,e^{i\theta}) : 0 \leq \theta \leq \frac{2\pi q}{p}\}.
\end{align*}
Also we define $-K_i$ to be the orientation reversal of $K_i$ for $i=1,2$. From this, it is clear that $\sigma(\pm K_i) = \pm K_{3-i}$ and $\tau(\pm K_i) = \mp K_i$ for $i = 1,2$. By Theorem~\ref{thm:mcg-lens} and~\ref{thm:mcg-lens-kernel}, we can determine when $\pm K_1$ and $\pm K_2$ become smoothly isotopic.

\begin{lemma}\label{lem:unknots}  
  The oriented rational unknots in $L(p,q)$ are given by 
  \[
  \begin{cases}
    K_1 \,& p = 2,  \\
    K_1,-K_1 \,& p \neq 2\, \text{ and }\, q \equiv \pm1 \modp p, \\
    K_1,-K_1,K_2,-K_2  & \text{otherwise}.
  \end{cases}
  \]
  up to smooth isotopy.
\end{lemma}

\begin{proof}
  Bonahon \cite{Bonahon:lens} essentially showed that if an orientation preserving diffeomorphism $f\colon L(p,q)\to L(p,q)$ sends $K_1$ to $K_2$ (up to isotopy), then it is smoothly isotopic to $\sigma$. He also showed that if $f$ sends $K_1$ to $-K_1$, then it is smoothly isotopic to $\tau$. Due to this fact, we only need to figure out when those diffeomorphisms are smoothly isotopic to the identity, which can be found in Theorem~\ref{thm:mcg-lens-kernel}. 

  If $p=2$, both $\sigma$ and $\tau$ are smoothly isotopic to the identity. Thus all $\pm K_1$ and $\pm K_2$ are smoothly isotopic.

  If $q \equiv -1$, $\sigma \circ \tau$ is smoothly isotopic to the identity. Thus $K_1$ is smoothly isotopic to $-K_2$, and $-K_1$ is smoothly isotopic to $K_2$.
  
  If $q \equiv 1$, $\sigma$ is smoothly isotopic to the identity, but $\tau$ is not. Thus $K_1$ is smoothly isotopic to $K_2$, and $-K_1$ is smoothly isotopic to $-K_2$.
  
  If $q \not\equiv \pm1$, none of $\sigma$ and $\tau$ is smoothly isotopic to the identity. Thus none of $\pm K_1$ and $\pm K_2$ is smoothly isotopic to each other.
\end{proof}

Let $(p,q)$ be a pair of coprime integers satisfying $p > q > 0$. Geiges and Onaran \cite{GO:unknots} depicted surgery presentations for the rational unknots $K_1$ and $K_2$ in $L(p,q)$, see Figure~\ref{fig:unknots} and \ref{fig:unknots-smooth}. Here, $p'/q'$ is the largest rational number satisfying $pq' - p'q = -1$.

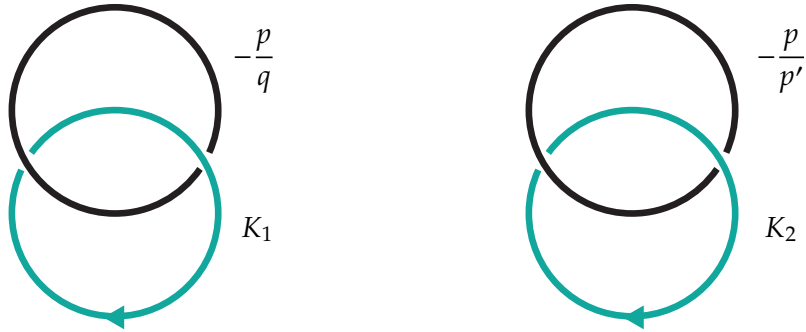
\begin{figure}[htbp]
  \vspace{0.2cm}
  \begin{overpic}[tics=20]{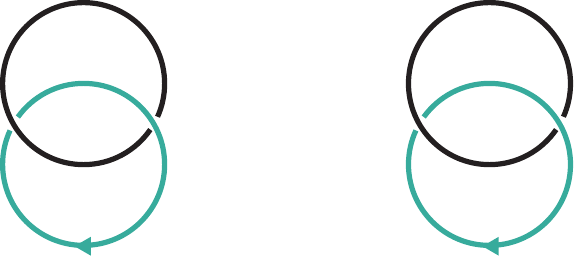}
    \put(85,100){$-\dfrac pq$}
    \put(88,35){$K_1$}

    \put(283,100){$-\dfrac{p}{p'}$}
    \put(286,35){$K_2$}
  \end{overpic}
  \vspace{0.1cm}
  \caption{Surgery presentations for the rational unknots $K_1$ and $K_2$.}
  \label{fig:unknots-smooth}
\end{figure}

We can write the negative continued fraction of $-p/q$ as follows: 
\[
  -\frac pq = [r_0, \dots, r_n]
\] 
where $r_i \leq -2$ for $0 \leq i \leq n$. Then we have 
\[ 
  \begin{pmatrix}
    p & p'\\
    -q & -q'
  \end{pmatrix} 
  =
  \begin{pmatrix}
    -r_0 & 1\\
    -1 & 0
  \end{pmatrix} 
  \begin{pmatrix}
    -r_1 & 1\\
    -1 & 0
  \end{pmatrix} 
  \cdots 
  \begin{pmatrix}
    -r_n & 1\\
    -1 & 0
  \end{pmatrix} 
\]
(see \cite{Rose:book} for example). After taking the inverse of these matrices, we obtain
\begin{equation}\label{eq:fraction}
  -\frac{p}{p'} = [r_n, \ldots, r_0].  
\end{equation}  
Notice that the two surgery presentations in Figure~\ref{fig:unknots-smooth} are not the same. However, by the equality (\ref{eq:fraction}), we can naturally identify one surgery presentation with the other as shown in Figure~\ref{fig:unknots-expansion}. Moreover, if we fix the signs of stabilization, then we can also identify one contact surgery presentation with the other in Figure~\ref{fig:unknots}.

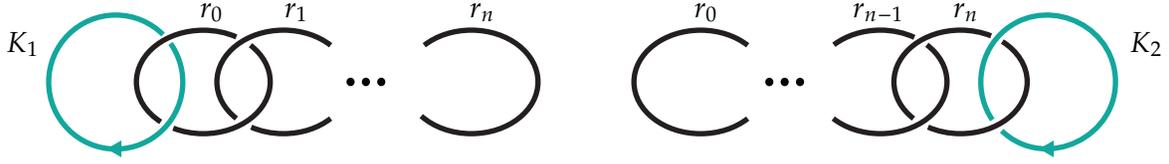
\begin{figure}[htbp]
  \vspace{0.35cm}
  \begin{overpic}[tics=20]{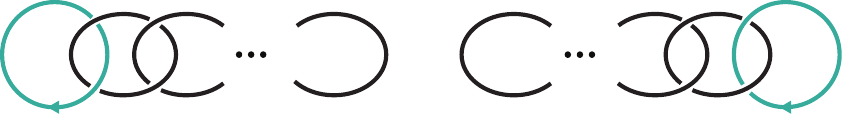}
    \put(-15,40){$K_1$}
    \put(58,53){$r_0$}
    \put(90,53){$r_1$}
    \put(160,53){$r_n$}

    \put(245,53){$r_0$}
    \put(305,53){$r_{n-1}$}
    \put(343,53){$r_n$}
    \put(410,40){$K_2$}
  \end{overpic}
  \vspace{0.05cm}
  \caption{Surgery presentations for the rational unknots $K_1$ and $K_2$.}
  \label{fig:unknots-expansion}
\end{figure}

We end this section by reviewing the mapping class group of $S^1 \times S^2$ and contactomorphisms of the standard contact structure on $S^1 \times S^2$. Consider $S^1 \times S^2 \subset S^1 \times \mathbb{R}^3$, where $S^2$ is a unit sphere in $\mathbb{R}^3$. Then the standard contact structure $\xi_{std}$ on $S^1 \times S^2$ is the kernel of 
\[
  \alpha = (z\,d\theta + x\,dy - y\,dx)|_{S^1 \times S^2}.
\]
There exists an orientation preserving diffeomorphism $\eta$ of $S^1 \times S^2$ which is defined by
\begin{align*}
  \eta\colon S^1 \times S^2 &\to S^1 \times S^2,\\ 
  (\theta,\mathbf{x}) &\mapsto (-\theta, -\mathbf{x}).\\
\end{align*}  
Consider a rotation matrix $r_\theta$ of $\mathbb{R}^3$ about the $z$-axis:
\[
  r_\theta = 
  \begin{pmatrix}
    \cos\theta & -\sin\theta & 0\\
    \sin\theta & \cos\theta & 0\\
    0 & 0 & 1
  \end{pmatrix}.
\]
There is another orientation preserving diffeomorphism $\delta$ of $S^1 \times S^2$, which is the Dehn twist about an essential sphere, defined by
\begin{align*}
  \delta\colon S^1 \times S^2 &\to S^1 \times S^2,\\ 
  (\theta,\mathbf{x}) &\mapsto (\theta, r_\theta(\mathbf{x})).
\end{align*}
We can check $\eta$ is a coorientation preserving contactomorphism immediately:
\begin{align*}
  \eta^*(\alpha) &= \eta^*(z\,d\theta + x\,dy - y\,dx)\\
  &= (-z)\,d(-\theta) - x\,d(-y) + y\,d(-x)\\
  &= \alpha.
\end{align*}
Although $\delta$ is not a contactomorphism of $\xi_{std}$, since there exists a unique tight contact structure on $S^1 \times S^2$ up to isotopy, two contact structures $\xi_{std}$ and $\delta_*(\xi_{std})$ are isotopic. We can apply the Moser's trick again and obtain an isotopy $\psi_t$ such that $(\psi_1)_*(\xi_{std}) = \delta_*(\xi_{std})$. Then clearly $\psi_1^{-1} \circ \delta$ is a contactomorphism of $\xi_{std}$ which is smoothly isotopic to $\delta$. We just relabel $\psi_1^{-1} \circ \delta$ as $\delta$.  

The mapping class group of $S^1 \times S^2$ was determined by Gluck \cite{Gluck:S1xS2}. After that, Hatcher \cite{Hatcher:S1xS2} determined the homotopy type of $\Diff(S^1 \times S^2)$.

\begin{theorem}[Gluck \cite{Gluck:S1xS2}, see also Hatcher \cite{Hatcher:S1xS2}]\label{thm:mcg-s1s2}
  The mapping class group of $S^1 \times S^2$ is 
  \[
    \pi_0(\Diff_+(S^1 \times S^2)) = \mathbb{Z}_2 \oplus \mathbb{Z}_2 \cong \langle \delta, \eta \rangle.
  \]
\end{theorem}

We define a \dfn{positively oriented core} of $S^1 \times S^2$ to be 
\[
  K := \{(\theta,p) : 0 \leq \theta \leq 2\pi \text{ and } p \in S^2\}, 
\]
and $-K$ to be its orientation reversal and call it a \dfn{negatively oriented core}. Notice that $K$ and $-K$ are not smoothly isotopic to each other. Chen, Ding and Li \cite{CDL:knots1s2} classified Legendrian representatives of the oriented core in the standard contact structure $\xi_{std}$ on $S^1 \times S^2$ up to Legendrian isotopy. Notice that the core is not (rationally) null-homologous, so the (rational) Thurston--Bennequin invariant is not well-defined. However, since $e(\xi_{std})=0$, the contact structure is trivial as a plane field, so the rotation number is well-defined once we fix a trivialization of $\xi_{std}$. The following theorem (and all other statements considering $S^1 \times S^2$) does not depend on the choice of the trivialization. 

\begin{theorem}[Chen--Ding--Li \cite{CDL:knots1s2}]\label{thm:unknot-s1s2}
  The oriented core in the standard contact structure $\xi_{std}$ on $S^1 \times S^2$ is Legendrian simple: any two Legendrian representatives are Legendrian isotopic if they have the same orientation and the same rotation number.
\end{theorem}

Ding and Geiges \cite{DG:S1xS2} showed that $\delta$ increases or decreases the rotation number of Legendrian representatives of the oriented cores in $S^1 \times S^2$. It is also not hard to see that $\eta$ changes the sign of the rotation number. 

\begin{lemma}[Ding--Geiges \cite{DG:S1xS2}]\label{lem:rot-s1s2}
  Let $L$ be a Legendrian representative of the positively oriented core in the standard contact structure $\xi_{std}$ on $S^1 \times S^2$ and $-L$ be its orientation reversal. Then we have 
  \[
    \rot(\delta(\pm L)) = \rot(\pm L) \pm 1.
  \]
  Also, we have
  \[
    \rot(\eta(\pm L)) = -\rot(\pm L).
  \]
\end{lemma}

%-------------------------------------------------------------------------------------------------------
\subsection{Invariants of rationally null-homologous Legendrian and transverse knots} \label{sec:tbq}
%-------------------------------------------------------------------------------------------------------
The classical invariants for null-homologous Legendrian and transverse knots were extended to rationally null-homologous knots by Baker and Etnyre \cite{BE:rational}. Let $L$ be a Legendrian representative of a rationally null-homologous knot in a contact $3$--manifold $(M,\xi)$. Suppose the order of $L$ in $H_1(M)$ is $r$ and a rational Seifert surface for $L$ is $\Sigma$. Let $L'$ be a push-off of $L$ along the contact framing. Then the rational Thurston--Bennequin invariant of $L$ is defined by 
\[
  \tb_{\mathbb{Q}}(L) := \frac 1r (L' \bigcdot \Sigma).
\]
Notice that there is an inclusion map $i\colon \Sigma \to M$ which is an embedding in the interior of $\Sigma$, and an $r$-fold cover of $L$ on $\partial \Sigma$. Since the pullback contact structure $i^*(\xi)$ is trivial as a plane field, the pullback of a non-vanishing tangent vector field $i^*(v)$ of $L$ gives a section on $\mathbb{R}^2 \times \partial\Sigma$ after fixing a trivialization. This induces a Gauss map $f \colon S^1 \to S^1$. We define the rational rotation number of $L$ as follows:
\[
  \rot_{\mathbb{Q}}(L) := \frac 1r \deg f. 
\]
Let $T$ be a transverse representative of a rationally null-homologous knot in a contact $3$--manifold $(M,\xi)$, and $L$ be a Legendrian representative such that $T$ is a (positive) transverse push-off of $L$. Then the rational self-linking number of $T$ is defined by
\[
  \self_{\mathbb{Q}}(T) := \tb_{\mathbb{Q}}(L) - \rot_{\mathbb{Q}}(L).
\]
We also denote the maximum rational Thurston--Bennequin invariant among the Legendrian representatives of $K$ by $\tbb_{\mathbb{Q}}(K)$. Similarly, we can define the maximum rational self-linking number $\selfb_{\mathbb{Q}}(K)$. 

Baker and Etnyre \cite{BE:rational} also showed that the stabilization has the same effect on the invariants as in the null-homologous case:
\begin{align*}
  \tb_{\mathbb{Q}}(S_\pm(L)) &= \tb_{\mathbb{Q}}(L) - 1,\\
  \rot_{\mathbb{Q}}(S_\pm(L)) &= \rot_{\mathbb{Q}}(L) \pm 1,\\
  \self_{\mathbb{Q}}(S(T)) &= \self_{\mathbb{Q}}(T) - 2.
\end{align*}

There are two ways to calculate the rational rotation number of a Legendrian rational unknot: using contact surgery presentations by Geiges and Onaran \cite{GO:unknots}, or using the Farey graph. We introduce both methods. 

It is well known that we can calculate the classical invariants from a contact surgery presentation for a given Legendrian knot in an integral homology sphere (see \cite{DK:invariants} for example). Geiges and Onaran \cite{GO:unknots} showed that the same formula works for contact surgery presentations for rationally null-homologous Legendrian knots in a homology sphere.  

Consider a contact surgery presentation for a rationally null-homologous Legendrian knot $L$ in a homology sphere. Convert the contact surgery presentation into a $(\pm1)$--surgery presentation. Let $L_1, \ldots, L_n$ be the surgery components of the $(\pm1)$--surgery presentation, $M$ be the linking matrix of $L_1,\ldots,L_n$ where the $i$-th diagonal entry is the smooth surgery coefficient of $L_i$, 
\[
  \mathbf{rot} := (\rot(L_1), \ldots, \rot(L_n))^\intercal
\]
where $\rot(L_i)$ is the rotation number of $L_i$ in $(S^3,\xi_{std})$,
\[
  \mathbf{lk} := \left(\mathrm{lk}(L,L_1),\ldots,\mathrm{lk}(L,L_n)\right)^\intercal
\]  
where $\mathrm{lk}(L,L_i)$ is the linking number between $L$ and $L_i$ and $\rot_0$ be the rotation number of $L$ in $(S^3,\xi_{std})$. 

\begin{lemma}[Geiges--Onaran \cite{GO:unknots}]\label{lem:rotq-surgery}
  With the notations defined above, we have
  \[
    \rot_{\mathbb{Q}}(L) = \rot_0 - \mathbf{rot}^\intercal \cdot M^{-1} \cdot \mathbf{lk}.
  \]
\end{lemma}

Notice that if we change the orientation of $L$, then $\rot_0$ changes the sign and every component in $\mathbf{lk}$ also changes the sign while $\mathbf{rot}$ and $M$ do not change. Thus we have $\rot_{\mathbb{Q}}(-L) = -\rot_{\mathbb{Q}}(L)$.

Now we introduce the second method. Recall from Section~\ref{sec:lens} that a decorated path $P$ for a tight contact structure on lens space $L(p,q)$ is the shortest path in the Farey graph from $-q/p$ to $0$, where all edges are decorated with $+$ or $-$ except for the first and the last ones. Let $-q/p = s_0, s_1, \ldots, s_n=-1$ be the vertices in $P$. Also, recall that $a/b \ominus c/d = (a-c)/(b-d)$. If $q \not\equiv -1 \pmod p$, we define  
\[
  r_1 = \sum_{i=1}^{n-1} \epsilon_i \left((s_{i} \ominus s_{i+1}) \bigcdot \frac {-p}q\right)
\]
and
\[
  r_2 = \sum_{i=1}^{n-1} \epsilon_i \left((s_{i+1} \ominus s_{i}) \bigcdot \frac 01\right)
\]
where $\epsilon_i$ is the sign of the edge from $s_i$ to $s_{i+1}$. Here, we assume the numerator of $s_i$ is negative and the denominator of $s_i$ is positive. If $q \equiv -1 \pmod p$, then we define both $r_1$ and $r_2$ to be $0$.

\begin{lemma}\label{lem:rotq-Farey}
  The Legendrian knots $L_1$ and $L_2$ in Figure~\ref{fig:unknots} have the rotation numbers 
  \[
    \rot_{\mathbb{Q}}(L_1) = \frac {r_1}p \quad \text{and} \quad \rot_{\mathbb{Q}}(L_2) = \frac {r_2}p.
  \]
\end{lemma}

\begin{proof}
  Recall that $L_1$ has the order $p$ in $H_1(L(p,q))$ and its standard neighborhood is $S^0(-1)$. Let $T = \partial S^0(-1)$ and $C = S_{-p/q}(-1)$ that is the complement $S^0(-1)$. Baker and Etnyre \cite{BE:rational} showed that the rational rotation number of $L$ is equal to 
  \[
    \rot_{\mathbb{Q}}(L_1)=\frac 1p e(\xi|_C,s)[D]
  \]
  where $s$ is a non-vanishing section of $\xi|_T$ and $D$ is a meridian disk of $C$. Decompose $C$ into 
  \[
    S_{-p/q}(s_1) \cup T(s_1,s_2) \cup \ldots \cup T(s_{n-1},s_n),
  \]
  where $T(s_i,s_{i+1})$ is a basic slice with slopes $s_i$ and $s_{i+1}$. According to \cite[Section~4.2]{Honda:classification1}, we can calculate the relative Euler class of a basic slice $T(s_i,s_{i+1})$ evaluated on a properly embedded annulus with $-p/q$ slope boundary as follows:
  \[
    e(\xi,t)[A] = \epsilon_i \left((s_{i} \ominus s_{i+1}) \bigcdot \frac {-p}q\right)
  \]
  where $t$ is a non-vanishing section of $\xi$ restricted to $\partial T(s_i,s_{i+1})$. Also, the relative Euler class of $S_{-p/q}(s_1)$ evaluates to $0$ on a meridian disk by Theorem~\ref{thm:classification-s1d2}. Since the relative Euler class is additive under union, we obtain the formula in the statement by taking a summation. The same argument works for $L_2$.
\end{proof}

Using these invariants, Baker and Etnyre \cite{BE:rational} coarsely classified Legendrian rational unknots in any tight contact structure on lens spaces. We say a knot type $K$ is \dfn{coarsely Legendrian simple} if for any two Legendrian representatives of $K$, there is a coorientation preserving contactomorphism, which is smoothly isotopic to the identity, sending one representative to the other if and only if their (rational) Thurston--Bennequin invariants and rotation numbers coincide.   

\begin{theorem}[Baker--Etnyre \cite{BE:rational}]\label{thm:unknot-coarse}
  Suppose $p > q > 0$ and $\xi$ is a tight contact structure on $L(p,q)$. Rational unknots in $\xi$ are coarsely Legendrian simple: there are Legendrian representatives 
  \[
  \begin{cases}
    L_1 \,& p = 2,  \\
    L_1,-L_1 \,& p \neq 2\, \text{ and }\, q \equiv \pm1 \modp p, \\
    L_1,-L_1,L_2,-L_2  & \text{otherwise}.
  \end{cases}
  \]
  with 
  \[
    \tb_{\mathbb{Q}}(\pm L_1)= -\frac{p-q}{p} \;\; \text{and}\;\; \tb_{\mathbb{Q}}(\pm L_2) = -\frac{p-p'}{p},\\
  \]
  where $p'/q'$ is the largest rational number satisfying $pq' - p'q = -1$. Also the rational rotation numbers are determined by the formula in Lemma~\ref{lem:rotq-surgery} or~\ref{lem:rotq-Farey}. For any Legendrian representative $L$ of rational unknots in $\xi$, there is a contactomorphism $f$ of $\xi$ which is smoothly isotopic to the identity such that $f(L)$ is one of the Legendrian representatives above, or their stabilization.
\end{theorem}

%%%%%%%%%%%%%%%%%%%%%%%%%%%%%%%%%%%%%%%%%%%%%%%%%%%%%%%%%%%%%%%%%%%%%%%%%%%%%%%%%%%%%%%%%%%%%%%%%%%
\section{Legendrian and transverse rational unknots in lens spaces} \label{sec:unknot-lens}
%%%%%%%%%%%%%%%%%%%%%%%%%%%%%%%%%%%%%%%%%%%%%%%%%%%%%%%%%%%%%%%%%%%%%%%%%%%%%%%%%%%%%%%%%%%%%%%%%%%

In this section, we classify Legendrian and transverse rational unknots in any tight contact structure on lens spaces and prove the theorems in Section~\ref{sec:unknot-intro}. To do so, we first determine the contact mapping class group of universally tight contact structures on a solid torus with two dividing curves. Then using this, we classify Legendrian representatives of the core in a tight contact structure on a solid torus with two dividing curves.

Before we start, we first extend the definitions. When $(M,\xi)$ is a contact manifold with convex boundary, we define 
\[
  \Cont(M,\xi) = \text{the group of contactomorphisms of $(M,\xi)$ that are the identity on $\partial M$}.
\]
Also we define the contact mapping class group of $(M,\xi)$ to be
\[
  \pi_0(\Cont(M,\xi)) = \Cont(M,\xi) / \sim
\]
where $f \sim g$ if $f$ is contact isotopic to $g$ relative to the boundary.

We start with determining the contact mapping class group of universally tight contact structures on $S^1 \times D^2$ with two dividing curves. If the dividing curves are longitudinal, it was already determined by Giroux~\cite{Giroux:cmcg} and Vogel~\cite{Vogel:pi1}. 

\begin{theorem}\label{thm:cmcg-s1d2}
  Let $\xi$ be a universally tight contact structure on a solid torus $V = S^1 \times D^2$ such that $\partial V$ is convex and the dividing set $\Gamma$ on $\partial V$ consists of two closed curves. Then we have 
  \[
    \pi_0(\Cont(V, \xi)) = 1.
  \]
\end{theorem}

\begin{proof}
  Let $D$ be a meridian disk of $V$. After isotopy, we can assume that $D$ is convex and $\partial D$ is Legendrian intersecting $\Gamma$ minimally. According to Theorem~\ref{thm:classification-s1d2}, the relative Euler class of $\xi$ is extremal, which implies that every dividing curve on $D$ is boundary parallel (it is called a \dfn{well-groomed} dividing set). See Figure~\ref{fig:univ} for example. Consider a bypass whose attaching arc lies on $D$. Notice that this bypass cannot be effective. Also, the attaching arc is one of the two configurations in Figure~\ref{fig:non-effective}, and the bypass is trivial or yields a contractible dividing curve. Since $\xi$ is tight, the bypass must be trivial. 
 
  \begin{figure}[htbp]
    \vspace{0.2cm}
    \begin{overpic}[tics=20]{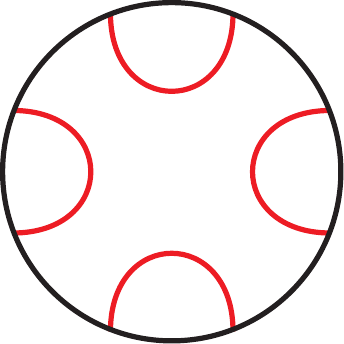}
    \end{overpic}
    \vspace{0.1cm}
    \caption{The dividing curves on a meridian disk in a universally tight $S^1 \times D^2$ with two dividing curves.}
    \label{fig:univ}
  \end{figure}

  Let $f \in \Cont(V,\xi)$. After a small perturbation, we can assume that $D$ and $f(D)$ intersect transversely in a finite set of circles. Choose an innermost circle $c$ among them. Then a disk in $D$ bounded by $c$ and a disk in $f(D)$ bounded by $c$ form a sphere. Since $V$ is irreducible, this sphere bounds a ball, so using this we can isotope the disk in $f(D)$ bounded by $c$ and reduce the number of intersection circles. See Figure~\ref{fig:intersection} for a schematic picture. Repeat this until $D$ and $f(D)$ intersect only in $\partial D$. Again, $D$ and $f(D)$ form a sphere and by irreducibility, this sphere bounds a ball. Thus $D$ and $f(D)$ are smoothly isotopic relative to the boundary. By Theorem~\ref{thm:discretization}, there exists a sequence of convex disks $D_1, \ldots, D_n$ with the identical boundary where $D_1 = D$, $D_n = f(D)$ and $D_{i+1}$ is obtained by attaching a bypass to $D_i$. As we observed above, the only allowable bypasses for $D$ are trivial bypasses, so inductively all $D_i$ has the same dividing set and $D_i$ and $D_{i+1}$ co-bound an $I$-invariant neighborhood by Lemma~\ref{lem:trivial}. Thus $D_i$ and $D_{i+1}$ are contact isotopic for $1 \leq i \leq n-1$ and this implies that $f$ is contact isotopic to a contactomorphism fixing $D$. 

  By Lemma~\ref{lem:fix}, we can further assume that $f$ fixes a small neighborhood $N$ of $\partial V \cup D$. Now pick a sphere $S$ contained in $N$ and parallel to a sphere $\partial N \setminus \partial V$. Perturb $S$ to be convex and let $B$ be the ball in $V$ bounded by $S$. By Eliashberg \cite[Thoerem~2.1.3]{Eliashberg:S3}, there exists a unique tight contact structure on $B$ up to isotopy fixing the characteristic foliation on $S$. Also, according to Eliashberg \cite[Theorem~2.4.2]{Eliashberg:S3}, we have $\pi_0(\Cont(B,\xi|_B)) = 1$. This implies that $f|_B$ is contact isotopic to the identity relative to the boundary. Since $f|_N$ is the identity, $f$ is contact isotopic to the identity relative to the boundary and this completes the proof.
\end{proof}

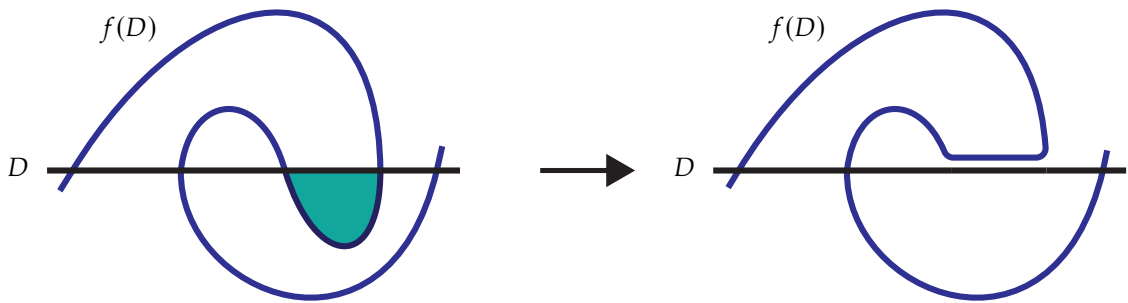
\begin{figure}[htbp]{\small
  \vspace{0.2cm}
  \begin{overpic}[tics=20]{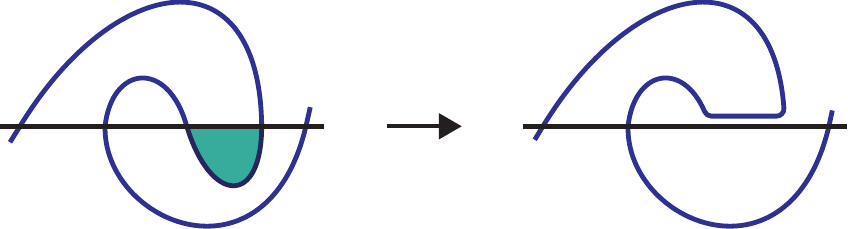}
    \put(-15,47){$D$}
    \put(20,100){$f(D)$}

    \put(237,47){$D$}
    \put(273,100){$f(D)$}
  \end{overpic}}
  \vspace{0.1cm}
  \caption{A schematic picture for $D$ and $f(D)$. A shaded region represents a ball bounded by two disks in $D$ and $f(D)$.}
  \label{fig:intersection}
\end{figure}

We need several steps to classify Legendrian and transverse rational unknots in tight contact structures on lens spaces. The first step is to classify Legendrian representatives of the core in a universally tight contact structure on a solid torus with two dividing curves. Legendrian knots in a solid torus with longitudinal dividing curves were already studied by Etnyre and V\'ertesi \cite{EV:pi1}. Recall that for a Legendrian knot $L$ and a given framing $F$ of $L$, the \dfn{twisting number} $\tw_F(L)$ is the difference between the framing induced by the contact structure and $F$. We note that $\tb(L) = \tw_{F}(L)$ when $F$ is the Seifert framing. Also, for a smooth knot type $K$,the \dfn{maximal twisting number} $\twb_F(K)$ is the maximal value of the twisting numbers with respect to $F$ among all Legendrian representatives of $K$.   

\begin{proposition}\label{prop:core-univ}
  Let $\xi$ be a universally tight contact structure on a solid torus $V = S^1 \times D^2$ with two dividing curves of slope $s$. Then the core of $(V,\xi)$ is Legendrian simple: there exists a unique Legendrian representative $L$ with the maximum twisting number $\twb_F = \lfloor s \rfloor$ where $F$ is the product framing of $V$. Any Legendrian representative of the core is Legendrian isotopic to $L$ or its stabilization. 
\end{proposition}

\begin{proof}
  We only consider the case $s \in [-1,0)$ since we can realize any dividing slope by the Dehn twists about a meridian disk. 
  
  We first show that there exists a unique Legendrian representative of the core of $(V,\xi)$ with the maximum twisting number $\twb_F=-1$ up to Legendrian isotopy. Suppose $L$ is a Legendrian representative of the core of $(V,\xi)$. The dividing slope of a standard neighborhood $N$ of $L$ is an integer, so $\lfloor s \rfloor = -1$ is the maximum twisting number. Let $L_1$ and $L_2$ be Legendrian representatives of the core with $\tw_F=-1$. Suppose $N_1$ and $N_2$ are standard neighborhoods of $L_1$ and $L_2$, respectively. Then we have $T_i(-1,s) = V \setminus N_i$ for $i=1,2$, which are minimally twisting $T^2 \times I$ layers with dividing slopes $-1$ and $s$. According to Theorem~\ref{thm:classification-s1d2}, a tight contact structure on $V$ is completely determined by the tight contact structure on $T_i(-1,s)$. Thus there exists a coorientation preserving contactomorphism $f \colon T_1(-1,s) \to T_2(-1,s)$ fixing $\partial V$. Since there exists a unique tight contact structure on a standard neighborhood of a Legendrian knot, we can extend $f$ to entire $(V,\xi)$ so that $f(L_1)=L_2$. By Theorem~\ref{thm:cmcg-s1d2}, $f$ is contact isotopic to the identity. Since $f$ sends $L_1$ to $L_2$, they are Legendrian isotopic.
  
  Next, we will show that if $L$ is a Legendrian representative of the core of $(V,\xi)$ with $n = \tw_F(L) < -1$, then $L$ destabilizes. Suppose $N$ is a standard neighborhood of $L$. Then we have $T(n,s) = V \setminus N$, which is a minimally twisting $T^2 \times I$ layer with the dividing slopes $n$ and $s$. Since $n<-1$, we can decompose $T(n,s)$ into $T(n,n+1) \cup T(n+1,s)$. Notice that $T(n,n+1)$ is a basic slice and we can thicken $N$ by attaching $T(n,n+1)$. This corresponds to a destabilization of $L$.
\end{proof}

Next, we improve the result by classifying the Legendrian representatives of the core in any tight contact structure on a solid torus with two dividing curves.  

\begin{proposition}\label{prop:core-any}
  Let $\xi$ be a tight contact structure on a solid torus $V = S^1 \times D^2$ with two dividing curves of slope $s$. Then the core of $(V,\xi)$ is Legendrian simple: there exists a unique Legendrian representative $L$ with the maximum twisting number $\twb_F = \lfloor s \rfloor$ where $F$ is the product framing of $V$. Any Legendrian representative of the core is Legendrian isotopic to $L$ or its stabilization. 
\end{proposition}

\begin{proof}
  Again, we only consider the case $s \in [-1,0)$ since we can realize any dividing slope by the Dehn twists about a meridian disk. 
 
  We first show that there exists a unique Legendrian representative of the core of $(V,\xi)$ with the maximum twisting number $\twb_F=-1$ up to Legendrian isotopy. Suppose $L$ is a Legendrian representative of the core of $(V,\xi)$. The dividing slope of a standard neighborhood $N$ of $L$ is an integer, so $\lfloor s \rfloor = -1$ is the maximum twisting number. Let $L_1$ and $L_2$ be Legendrian representatives of the core with $\tw_F=-1$. Take a meridian disk $D$ of $V$ intersecting $L_2$ transversely once. Perturb $D$ to be convex with Legendrian boundary such that $\partial D$ intersects $\Gamma_{\partial V}$ minimally and $D$ intersects $L_1$ transversely. We will consider two cases according to the intersection number between $L_1$ and $D$.
 
  First, we consider the case  $|D \cap L_1| = 1$. Suppose $N_1$ and $N_2$ are standard neighborhoods of $L_1$ and $L_2$, respectively. After perturbing $D$ and $\partial N_i$, we can assume that the ruling slope of $N_i$ is $\infty$ and there exists a ruling curve $c_i$ that lies on $D$ and it is the only intersection between $\partial N_i$ and $D$ for $i = 1,2$. Since $\tb(c_1)=\tb(c_2)=-1$ by the equality (\ref{eq:tb}), each $c_1$ and $c_2$ intersects a dividing curve on $D$ at two points. See Figure~\ref{fig:disk1} for example. Choose components $d_1,\ldots,d_n$ of the dividing set of $D$ such that $c_1$ intersects $d_1$, $c_2$ intersects $d_n$, and $d_i$ and $d_{i+1}$ are adjacent. We claim that we can isotope $L_1$ through Legendrian knots so that $c_1$ intersects $d_2$ and does not intersect any other dividing curve on $D$. Take a solid torus $\overline{N}$ such that $\overline{N}$ contains $N_1$ and $\partial \overline{N}$ intersects $D$ in a closed curve $\overline{c}$ that contains $c_1$ and intersects $d_1$ and $d_2$ at two points each. See Figure~\ref{fig:disk2} for example. Perturb $\partial \overline{N}$ to be convex and $\overline{c}$ to be Legendrian. Let $\overline{s}$ be the dividing slope of $\overline{N}$. By the equality (\ref{eq:tb}), we have $\tb(\overline{c})=-2$. Due to this fact, there are only three cases we need to consider for the dividing curves on $\partial \overline{N}$.  

  The first case is $\overline{s} > -1$. Let $2n$ be the number of dividing curves on $\partial \overline{N}$ and $\overline{s} = p/q$ for $|p| > q \geq1$. Since the dividing set interleaves, $\left|\overline{c} \cap \Gamma_D\right| = \left|\overline{c} \cap \Gamma_{\partial \overline{N}}\right|$. Thus we have 
  \[
    \tb(\overline{c}) = -2 = -\frac12\left|\overline{c} \cap \Gamma_D\right| = -\frac12\left|\overline{c} \cap \Gamma_{\partial \overline{N}}\right| \leq -n\left|\frac pq \bigcdot \frac 10\right| = -nq. 
  \]
  The equality holds if and only if $\overline{c}$ intersects $\Gamma_{\partial \overline{N}}$ minimally. Since $-1 < \overline{s} \leq s \in (-1,0)$, we have $q > 1$ and this implies that $n = 1$. Thus there are two dividing curves on $\partial \overline{N}$. Notice that the disk $\overline{D} \subset D$, bounded by $\overline{c}$, contains two boundary-parallel dividing curves as shown in Figure~\ref{fig:disk2}. Notice that these two dividing curves are a part of $d_1$ and $d_2$, but we just relabel them as $d_1$ and $d_2$. According to Theorem~\ref{thm:imba}, we can take a bypass lying on $\overline{D}$ containing the dividing curve $d_1$. Remove a bypass attachment of this bypass from $\overline{N}$. Then by Theorem~\ref{thm:bypass-torus}, the resulting solid torus $\overline{N}_1$ has two dividing curves of slope $\overline{s}_1$ satisfying $-1 \leq \overline{s}_1 < \overline{s}$, and the resulting meridian disk $\overline{D}_1$ contains the single dividing curve $d_2$. Perturb $\overline{c}_1 =\partial \overline{D}_1$ to be Legendrian. Then by equality (\ref{eq:tb}), we have $\tb(\overline{c}_1) = -1$. Since the dividing set interleaves, we have 
  \[
    \tb(\overline{c}) = -1 \leq -\left|\overline{s}_1 \bigcdot \frac 10\right|,
  \]
  which implies that $\overline{s}_1$ is an integer. Thus $\overline{s}_1 = -1$ and $\overline{N}_1$ has two dividing curves of slope $-1$. Let $\overline{L}_1$ be a Legendrian representative of the core of $\overline{N}_1$ with $\tw_F=-1$. Then $\overline{N}_1$ is a standard neighborhood of $\overline{L}_1$. Since $\overline{D}$ only contains boundary-parallel dividing curves, the restricted contact structure $\xi|_{\overline{N}}$ is universally tight by Theorem~\ref{thm:classification-s1d2}. Since $\overline{N}$ contains both $L_1$ and $\overline{L}_1$, by Proposition~\ref{prop:core-univ}, $L_1$ is Legendrian isotopic to $\overline{L}_1$. Notice that $\overline{c}_1$ intersects $d_2$ and does not intersect any other dividing curve on $D$. 
 
  \begin{figure}[htbp]
    \vspace{0.2cm}
    \begin{overpic}[tics=20]{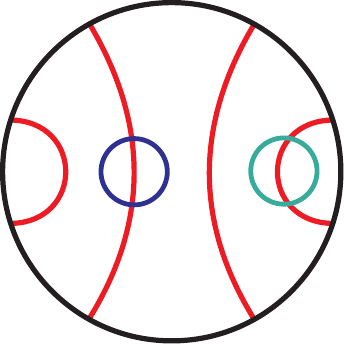}
    \end{overpic}
    \vspace{0.1cm}
    \caption{The red curves are the dividing curves on a convex disk $D$. The closed curves are Legendrian curves.}
    \label{fig:disk1}
  \end{figure}   

  The second case is $\overline{s} = -1$ and there are four dividing curves on $\partial \overline{N}$. Again, the disk $\overline{D} \subset D$, bounded by $\overline{c}$, contains two boundary parallel dividing curves as shown in Figure~\ref{fig:disk2}. According to Theorem~\ref{thm:imba}, we can take a bypass lying on $\overline{D}$ containing the dividing curve $d_1$. Remove a bypass attachment of this bypass from $\overline{N}$ and let $\overline{N}_1$ be the resulting solid torus and $\overline{D}_1$ be the resulting meridian disk. Perturb $\overline{c}_1 = \partial \overline{D}_1$ to be Legendrian. Since $\overline{D}_1$ contains the single dividing curve $d_2$, we have $\tb(\overline{c}_1) = -1$. Thus there are two dividing curves on $\partial \overline{N}_1$ as discussed in the first case. Since there are more than two dividing curves on $\partial \overline{N}$, the bypass attachment does not change the dividing slope. Thus $\overline{N}_1$ has two dividing curves of slope $-1$ and $\overline{N} \setminus \overline{N}_1$ is a non-rotative layer. Let $\overline{L}_1$ be a Legendrian representative of the core of $\overline{N}_1$ with $\tw_F = -1$. Then $\overline{N}_1$ is a standard neighborhood of $\overline{L}_1$. By the attach=dig principle (Theorem~\ref{thm:attach=dig}), there is a solid torus $\widetilde{N}$ containing $\overline{N}$ with two dividing curves of slope $-1$. By Theorem~\ref{thm:classification-s1d2}, the restricted contact structure $\xi|_{\widetilde{N}}$ is universally tight. Since $\widetilde{N}$ contains both $L_1$ and $\overline{L}_1$, by Proposition~\ref{prop:core-univ}, $L_1$ is Legendrian isotopic to $\overline{L}_1$. Notice that $\overline{c}_1$ intersects $d_2$ and does not intersect any other dividing curve on $D$. 
 
  \begin{figure}[htbp]
    \vspace{0.2cm}
    \begin{overpic}[tics=20]{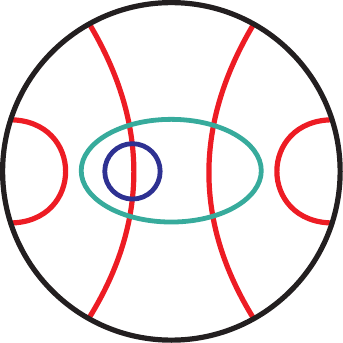}
    \end{overpic}
    \vspace{0.1cm}
    \caption{The red curves are the dividing curves on a convex disk $D$. The closed curves are Legendrian curves.}
    \label{fig:disk2}
  \end{figure}

  The third case is $\overline{s} = -1$ and there are two dividing curves on $\partial \overline{N}$. In this case, $\overline{c}$ does not intersect $\Gamma_{\partial \overline{N}}$ minimally. However, the disk $\overline{D} \subset D$, bounded by $\overline{c}$, still contains two boundary parallel dividing curves as shown in Figure~\ref{fig:disk2}. According to Theorem~\ref{thm:imba}, we can take a bypass lying on $\overline{D}$ containing the dividing curve $d_1$. Remove a bypass attachment of this bypass from $\overline{N}$ and let $\overline{N}_1$ be the resulting solid torus and $\overline{D}_1$ be the resulting meridian disk. Perturb $\overline{c}_1 = \partial \overline{D}_1$ to be Legendrian. Since $\overline{D}_1$ contains the single dividing curve $d_2$, we have $\tb(\overline{c}_1) = -1$. Thus there are two dividing curves on $\partial \overline{N}_1$ as discussed in the first case. Since $\overline{c}$ does not intersect $\Gamma_{\partial \overline{N}}$ minimally, the bypass is not effective and the bypass attachment does not change the dividing slope. Thus $\overline{N}_1$ has two dividing curves of slope $-1$. Let $\overline{L}_1$ be a Legendrian representative of the core of $\overline{N}_1$ with $\tw_F = -1$. Then $\overline{N}_1$ is a standard neighborhood of $\overline{L}_1$. By Theorem~\ref{thm:classification-s1d2}, the restricted contact structure $\xi|_{\overline{N}}$ is universally tight. Since $\overline{N}$ contains both $L_1$ and $\overline{L}_1$, by Proposition~\ref{prop:core-univ}, $L_1$ is Legendrian isotopic to $\overline{L}_1$. Notice that $\overline{c}_1$ intersects $d_2$ and does not intersect any other dividing curve on $D$.
  
  We just have proved the claim. By applying the claim inductively, we can isotope $L_1$ through Legendrian knots until $c_1$ intersects $d_n$ and does not intersect any other dividing curve on $D$. After that, take a bypass lying on $D$ which does not contain $d_n$, and remove a bypass attachment of the bypass from $V$. Repeat this until there is only one dividing curve, $d_n$, left. See Figure~\ref{fig:disk3} for example. Let $\overline{V}$ be the resulting solid torus and $\overline{D}$ be the resulting meridian disk. Perturb $\overline{c} = \partial \overline{D}$ to be Legendrian. Since $\overline{D}$ contains the single dividing curve $d_n$, we have $\tb(\overline{c}) = -1$. Let $\overline{s} = p/q$ be the dividing slope of $\partial \overline{V}$ and $2n$ be the number of dividing curves. Since the dividing set interleaves, we have
  \[
     \tb(\overline{c}) = -1 \leq -n\left|\frac pq \bigcdot \frac 10\right|.
  \] 
  The equality holds if and only if $\overline{c}$ intersects $\Gamma_{\partial V}$ minimally. From the inequality, we have $n=1$ and $q=1$. Thus there are two dividing curves on $\partial \overline{V}$ and $\overline{s}$ is an integer. Since none of the bypasses intersects both $N_1$ and $N_2$ and the bypass attachment is a local operation, $\overline{V}$ contains both $N_1$ and $N_2$. Thus we have $-1 \leq \overline{s} \leq s \in [-1,0)$ and this implies that $\overline{s}=-1$. Thus $\overline{V}$ has two dividing curves of slope $-1$. By Theorem~\ref{thm:classification-s1d2}, the restricted contact structure $\xi_{\overline{V}}$ is universally tight. Since $\overline{V}$ contains both $L_1$ and $L_2$, by Proposition~\ref{prop:core-univ}, $L_1$ and $L_2$ are Legendrian isotopic.

  \begin{figure}[htbp]
    \vspace{0.2cm}
    \begin{overpic}[tics=20]{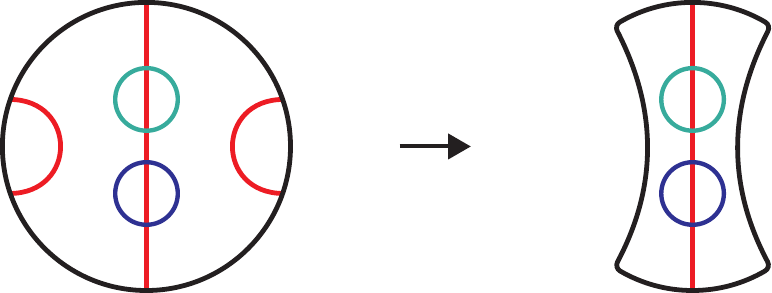}
    \end{overpic}
    \vspace{0.1cm}
    \caption{The red curves are the dividing curves on a convex disk $D$. The closed curves are Legendrian curves.}
    \label{fig:disk3}
  \end{figure}

  Next, we consider the case $m := |D \cap L_1| > 1$. In this case, we can perturb $\partial N_1$ so that the ruling slope is $\infty$ and there are $m$ ruling curves $c_1^1, \ldots, c_m^1$ lying on $D$ and each $c_i^1$ intersects a dividing curve on $D$ at two points. Similarly, we can also perturb $\partial N_2$ so that the ruling slope is $\infty$ and there is a ruling curve $c^2$ lying on $D$ intersecting a dividing curve on $D$ at two points. Choose components $d_1,\ldots,d_n$ of the dividing set of $D$ such that $c_1^1$ intersects $d_1$, $c^2$ intersects $d_n$, and $d_i$ and $d_{i+1}$ are adjacent. We claim that we can isotope $L_1$ through Legendrian knots so that $c_1^1$ intersects $d_2$, while fixing other $c_i^1$ for $2 \leq i \leq m$. After perturbing $D$, we can take a solid torus $\overline{N}$ such that $\overline{N}$ contains $N_1$ and there are $m$ ruling curves $\overline{c}_1,\ldots,\overline{c}_m$ of $\partial \overline{N}$ lying on $D$ such that $\overline{c}_i = c_i^1$ for $2 \leq i \leq m$, the disk bounded by $\overline{c}_1$ contains $c_1^1$ and $\overline{c}_1$ intersects $d_1$ and $d_2$ at four points. By the equality (\ref{eq:tb}), we have $\tb(\overline{c}_2) = -1$ and this implies that there are two dividing curves on $\partial \overline{N}$ and the dividing slope $\overline{s}$ is an integer as discussed above. Also, since $\overline{N}$ contains $N_1$, we have $-1 \leq \overline{s} \leq s \in [-1,0)$ and $\overline{s} = -1$. Thus $\overline{N}$ has two dividing curves of slope $-1$. This implies that $\overline{c}_1$ does not intersect $\Gamma_{\partial \overline{N}}$ minimally. Since there are two boundary-parallel dividing curves on the disk $\overline{D} \subset D$, bounded by $\overline{c}_1$, we can find a bypass lying on $\overline{D}$ that contains $d_1$ according to Theorem~\ref{thm:imba}. Remove a bypass attachment of this bypass from $\overline{N}$ and let $\overline{N}_1$ be the resulting solid torus and $\overline{D}_1$ be the resulting meridian disk. Since $\overline{D}_1$ contains the single dividing curve $d_2$, there are still two dividing curves on $\overline{N}_1$. Since $\overline{c}_1$ does not intersect $\Gamma_{\partial \overline{N}}$ minimally, the bypass is not effective and the bypass attachment does not change the dividing slope. Thus $\overline{N}_1$ has two dividing curves of slope $-1$. Let $\overline{L}_1$ be a Legendrian representative of the core of $\overline{N}_1$ with $\tw_F = -1$. Then $\overline{N}_1$ is a standard neighborhood of $\overline{L}_1$. By Theorem~\ref{thm:classification-s1d2}, the restricted contact structure $\xi|_{\overline{N}}$ is universally tight. Since $\overline{N}$ contains both $L_1$ and $\overline{L}_1$, by Proposition~\ref{prop:core-univ}, $L_1$ is Legendrian isotopic to $\overline{L}_1$. This completes the claim. 
  
  By applying the claim inductively, we can isotope $L_1$ through Legendrian knots until $c_1^1$ intersects $d_n$ while fixing other $c_i^1$ for $2 \leq i \leq m$. After that, apply the claim to $c_2^1$ and we can isotope $L_1$ through Legendrian knots until $c_2^1$ intersects $d_n$ while fixing other $c_i^1$. Repeat the argument until all $c_i^1$ for $1 \leq i \leq m$ intersect $d_n$. Now using Theorem~\ref{thm:imba}, take a bypass lying on $D$ that does not contain $d_n$ and remove a bypass attachment of this bypass from $V$. Repeat this until there is only one dividing curve, $d_n$, left. Let $\overline{V}$ be the resulting solid torus and $\overline{D}$ be the resulting meridian disk. Perturb $\overline{c} = \partial\overline{D}$ to be Legendrian. Since $\overline{D}$ contains the single dividing curve $d_n$, we have $\tb(\overline{c}) = -1$ and this implies that there are two dividing curves on $\partial \overline{V}$ and the dividing slope $\overline{s}$ is an integer as discussed above. Since none of the bypasses intersects both $N_1$ and $N_2$ and the bypass attachment is a local operation, $\overline{V}$ contains both $N_1$ and $N_2$. Thus we have $-1 \leq \overline{s} \leq s \in [-1,0)$ and this implies that $\overline{s}=-1$. Thus $\overline{V}$ has two dividing curves of slope $-1$. By Theorem~\ref{thm:classification-s1d2}, the restricted contact structure $\xi_{\overline{V}}$ is universally tight. Since $\overline{V}$ contains both $L_1$ and $L_2$, by Proposition~\ref{prop:core-univ}, $L_1$ and $L_2$ are Legendrian isotopic. 
  
  Lastly, we show that if $L$ is a Legendrian representative of the core with $n := \tw_F < -1$, then $L$ destabilizes. Suppose $N$ is a standard neighborhood of $L$. Then we have $V \setminus N = T(n,s)$, which is a minimally twisting $T^2 \times I$ layer with the dividing slopes $n$ and $s$. Since $n<-1$, we can decompose $T(n,s)$ into $T(n,n+1) \cup T(n+1,s)$. Notice that $T(n,n+1)$ is a basic slice and we can thicken $N$ by attaching $T(n,n+1)$. This corresponds to a destabilization of $L$.
\end{proof}

Before we classify Legendrian and transverse rational unknots, we need one additional lemma.

\begin{lemma} \label{lem:nbhd}
  Let $L$ be a Legendrian rational unknot in a tight contact structure on $L(p,q)$, which is smoothly isotopic to $K_1$ and $N(L)$ a standard neighborhood of $L$. If $\tb_{\mathbb{Q}}(L) = \tbb_{\mathbb{Q}}(K_1)$, then there is a decomposition $L(p,q) = S_{-p/q}(-1) \cup S^0(-1)$ such that $N(L)=S^0(-1)$.
\end{lemma}

\begin{proof}
  Recall from Section~\ref{sec:lens} and~\ref{sec:mcg-lens} that any tight contact structure on $L(p,q)$ can be decomposed into $S_{-p/q}(-1)$ and $S^0(-1)$, and $K_1$ is the core of $S^0(-1)$ and $K_2$ is the core $S_{-p/q}(-1)$. Also, we showed that there is a Legendrian representative $L_1$ of $K_1$, whose standard neighborhood is $S^0(-1)$, see Figure~\ref{fig:unknots}. By Theorem~\ref{thm:classification-lens}, this decomposition is unique, so there exists a unique $L_1$ up to contactomorphism. We will show $\tb_{\mathbb{Q}}(L_1) = \tbb_{\mathbb{Q}}(K_1)$. This will imply that $L_1$ is coarsely equivalent to $L$ and proves the lemma. Let $N$ be a standard neighborhood of a Legendrian representative of $K_1$. Notice that $N$ has longitudinal dividing curves of $S^0$, which implies there is an edge between the dividing slope $s$ and $0$ in the Farey graph, so $s=1/n$ for some $n \in \mathbb{Z}$. If $n \geq 0$, then a non-minimally twisting $T^2 \times I$ layer embeds in $L(p,q)$, which contradicts the tightness of $\xi$. Thus $n \leq -1$ and $N$ is $S^0(1/n)$ and the complement of $N$ is $S_{-p/q}(1/n)$. Clearly the largest solid torus among $S^0(1/n)$ is $S^0(-1)$ and it contains all other $S^0(1/n)$. This implies that $L_1$ has the maximal rational Thurston--Bennequin invariant.
\end{proof}

Now we are ready to classify Legendrian and transverse rational unknots in any tight contact structure on lens spaces. We first show the Legendrian simplicity. 

\begin{proposition}\label{prop:unknot}
  Let $\xi$ be a tight contact structure on a lens space $L(p,q)$ and $K$ be an oriented rational unknot in $L(p,q)$. Then there exists a unique Legendrian representative $L$ of $K$ in $\xi$ such that any Legendrian representative of $K$ is Legendrian isotopic to $L$ or its stabilization.
\end{proposition}

\begin{proof}
  Recall from Section~\ref{sec:lens} that a tight contact structure on $\xi$ on $L(p,q)$ can be decomposed into $S_{-p/q}(-1) \cup S^0(-1)$. Also, recall that $s_0,\ldots,s_n$ are the vertices of the shortest path in the Farey graph where $s_0 = -p/q$ and $s_n = 0$. Thus $\xi$ can also be decomposed into $S_{-p/q}(s_1)$ and $S^0(s_1)$. Suppose $K = K_1$. We first show that there exists a unique Legendrian representative of $K$ with $\tbb_{\mathbb{Q}}(K)$ up to Legendrian isotopy.
  
  Let $L$ and $L'$ be Legendrian representatives of $K$ with $\tb_{\mathbb{Q}}(L) = \tb_{\mathbb{Q}}(L') = \tbb_{\mathbb{Q}}(K)$, and $N$ and $N'$ be standard neighborhoods of $L$ and $L'$, respectively. By Lemma~\ref{lem:nbhd}, both $N$ and $N'$ are $S^0(-1)$, i.e., they have two dividing curves of slope $-1$ with the upper meridional slope $0$. Since $L$ and $L'$ are smoothly isotopic, there exists a smooth isotopy from $N$ to $N'$. Then by Theorem~\ref{thm:discretization}, there exists a sequence of solid tori $N_1, \ldots, N_n$ where $N_1 = N$, $N_n = N'$ and $N_{i+1}$ is obtained by attaching or removing a bypass to $\partial N_i$. Let $s_i$ be the dividing slope of $N_i$. Here, we define Legendrian representatives $L_i$ associated to $N_i$ as follows. First, if $s_i \leq -1$, then $N_i$ contains a solid torus $S^0(-1)$. Define $L_i$ to be a Legendrian representative of the core of this $S^0(-1)$ with the maximum twisting number. Notice that this $S^0(-1)$ is a standard neighborhood of $L_i$ so $\tb_{\mathbb{Q}}(L_i) = \tbb_{\mathbb{Q}}(K)$. If $s_i > -1$, then $N_i$ is contained in some $S^0(-1)$. Define $L_i$ to be a Legendrian representative of the core of this $S^0(-1)$ with the maximum twisting number. Notice that this $S^0(-1)$ is a standard neighborhood of $L_i$ so $\tb_{\mathbb{Q}}(L_i) = \tbb_{\mathbb{Q}}(K)$. From the definition, we can choose $L_0$ to be $L$ and $L_n$ to be $L'$. We claim that $L_i$ and $L_{i+1}$ are Legendrian isotopic, and this implies that $L$ and $L'$ are Legendrian isotopic by induction. Observe that if $s_i > -1$, then $s_{i+1} \geq -1$ by Theorem~\ref{thm:bypass-torus}. Similarly, if $s_i < -1$, then $s_{i+1} \leq -1$. Due to this fact, there are only two cases we need to consider. 
  
  The first case is $s_i, s_{i+1} \leq -1$. Assume that $N_{i+1}$ is obtained by attaching a bypass to $\partial N_i$ which is contained in $N_i$. In this case, $N_i$ contains $N_{i+1}$ and this implies that $N_i$ contains both $L_i$ and $L_{i+1}$. If $\partial N_i$ has more than two dividing curves, then by the attach=dig principle (Theorem~\ref{thm:attach=dig}), we can thicken $N_i$ and reduce the number of dividing curves. Also notice that $L_i$ and $L_{i+1}$ still have the maximum twisting number in $N_i$. If not, they destabilize and it contradicts that they have $\tbb_{\mathbb{Q}}(K)$. Thus by Proposition~\ref{prop:core-any}, $L_i$ and $L_{i+1}$ are Legendrian isotopic. Now assume that $N_{i+1}$ is obtained by attaching a bypass to $\partial N_i$ which is not contained in $N_i$. Then $N_{i+1}$ contains $N_i$ and this implies that $N_{i+1}$ contains both $L_i$ and $L_{i+1}$. Thus we can apply the same argument above (by switching the role of $N_i$ and $N_{i+1}$) and conclude that $L_i$ and $L_{i+1}$ are Legendrian isotopic.
 
  The second case is $s_i, s_{i+1} \geq -1$. In this case, the complements of each $N_i$ and $N_{i+1}$ contains $S_{-p/q}(s_1)$ since $s_1 \leq -1$. Let $\overline{L}_i$ and $\overline{L}_{i+1}$ be the Legendrian representatives of the core of each $S_{-p/q}(s_1)$ contained in the complements of $N_i$ and $N_{i+1}$, respectively, with the maximum twisting number. Notice that standard neighborhoods of $\overline{L}_i$ and $\overline{L}_{i+1}$ are $S_{-p/q}(s_1)$. Assume that $N_{i+1}$ is obtained by attaching a bypass to $\partial N_i$ which is contained in $N_i$. Then $N_i$ contains $N_{i+1}$ and this implies that the complement of $N_{i+1}$ contains both $\overline{L}_i$ and $\overline{L}_{i+1}$. Also notice that $\overline{L}_i$ and $\overline{L}_{i+1}$ still have the maximum twisting number in the complement of $N_{i+1}$. If not, they destabilize and standard neighborhoods of them are $S_{-p/q}(s)$ where $s$ is clockwise of $s_1$ and there is an edge between $s$ and $-p/q$ in the Farey graph. This implies that $s < -p/q$ or $s = \infty$, and a non-minimally twisting $T^2 \times I$ layer embeds in $(L(p,q),\xi)$, which contradicts the tightness of $\xi$. Now $\overline{L}_i$ and $\overline{L}_{i+1}$ are Legendrian isotopic by Proposition~\ref{prop:core-any}. Thus after Legendrian isotopy, we can identify $\overline{L}_i$ with $\overline{L}_{i+1}$, and then both $L_i$ and $L_{i+1}$ are contained in the complement of a standard neighborhood of $\overline{L}_i$. By Proposition~\ref{prop:core-any} again, $L_i$ and $L_{i+1}$ are Legendrian isotopic. Now assume that $N_{i+1}$ is obtained by attaching a bypass to $\partial N_i$ which is not contained in $N_i$. Then $N_{i+1}$ contains $N_i$ and this implies that the complement of $N_i$ contains both $\overline{L}_i$ and $\overline{L}_{i+1}$. Thus we can apply the same argument above (by switching the role of $N_i$ and $N_{i+1}$) and conclude that $L_i$ and $L_{i+1}$ are Legendrian isotopic. This completes the claim.
  
  By Theorem~\ref{thm:unknot-coarse}, a Legendrian representative $L$ of $K$ with $\tb_{\mathbb{Q}}(L) < \tbb_{\mathbb{Q}}(K)$ destabilizes. 

  The identical argument works for $-K_1$ and $\pm K_2$. We leave them as exercises for the reader. 
\end{proof}

\begin{remark}
  Notice that in the proof of Proposition~\ref{prop:unknot}, we need Proposition~\ref{prop:core-any}, not just Proposition~\ref{prop:core-univ} even for the universally tight contact structures on $L(p,q)$. This is because there exist virtually overtwisted neighborhoods of $K$ even in the universally tight contact structures.
\end{remark}

\begin{proof}[Proof of Theorem~\ref{thm:unknot-Legendrian}]
  The theorem immediately follows from Theorem~\ref{thm:unknot-coarse} and Proposition~\ref{prop:unknot}.
\end{proof}

\begin{proof}[Proof of Theorem~\ref{thm:unknot-transverse}]
  In \cite[Proof of Theorem~2.10]{EH:transverse}, Etnyre and Honda showed that the classification of transverse knots is equivalent to the classification of Legendrian knots up to negative stabilization. Thus the theorem immediately follows from Theorem~\ref{thm:unknot-Legendrian}.
\end{proof}

%%%%%%%%%%%%%%%%%%%%%%%%%%%%%%%%%%%%%%%%%%%%%%%%%%%%%%%%%%%%%%%%%%%%%%%%%%%%%%%%%%%%%%%%%%%%%%%%%
\section{The contact mapping class group of the standard lens spaces} \label{sec:cmcg-lens}
%%%%%%%%%%%%%%%%%%%%%%%%%%%%%%%%%%%%%%%%%%%%%%%%%%%%%%%%%%%%%%%%%%%%%%%%%%%%%%%%%%%%%%%%%%%%%%%%%

In this section, we use the results from the previous sections to prove Theorem~\ref{thm:cmcg-lens}, and Corollary~\ref{cor:inclusion} and~\ref{cor:cont0}. We also prove Theorem~\ref{thm:cmcg-s1s2} using the results of Ding and Geiges \cite{DG:S1xS2}, see Section~\ref{sec:mcg-lens}.   

\begin{proof}[Proof of Theorem~\ref{thm:cmcg-lens}] 
  Recall from Section~\ref{sec:lens} that the standard contact structure $\xi_{std}$ on $L(p,q)$ can be decomposed into $S_{-p/q}(-1) \cup S^0(-1)$, and the contact structure $\xi_{std}$ restricted to $S^0(-1)$ is universally tight by Theorem~\ref{thm:classification-s1d2}. Also, the contact structure $\xi_{std}$ restricted to $S_{-p/q}(-1)$ is universally tight by Theorem~\ref{thm:classification-lens}. Let $L_1$ be a Legendrian representative of $K_1$ with $\tbb_{\mathbb{Q}}(K_1)$. By Lemma~\ref{lem:nbhd}, a standard neighborhood of $L_1$ is $S^0(-1)$. 
  
  Recall from Section~\ref{sec:lens} that if $q^2 \equiv 1 \pmod p$, then $\sigma$ is a contactomorphism of the standard contact structure $\xi_{std}$ on $L(p,q)$. Also, if $q \equiv -1 \pmod p$, then $\tau$ is smoothly isotopic to a contactomorphism $\overline{\tau}$ of $\xi_{std}$ on $L(p,q)$. If $q \not\equiv -1 \pmod p$, then there exist two standard contact structure $\xi_{std}^\pm$ on $L(p,q)$ and $\tau_*$ sends one to the other. 
  \begin{claim}\label{claim:contacto}
    Any contactomorphism $f \in \Cont(L(p,q),\xi_{std})$ is contact isotopic to either $\sigma$, $\overline{\tau}$, or the identity.
  \end{claim} 
  To prove the claim, it is enough to show that that the natural inclusion 
  \[
    i_*\colon \pi_0(\Cont(L(p,q), \xi_{std})) \to \pi_0(\Diff_+(L(p,q)))
  \] 
  is injective. Suppose $f$ is smoothly isotopic to the identity. Since $f$ is a contactomorphism, $f$ sends a Legendrian representative of $K_1$ with $\tbb_{\mathbb{Q}}(K_1)$ to the one with $\tbb_{\mathbb{Q}}(K_1)$. Then by Proposition~\ref{prop:unknot} (or Theorem~\ref{thm:unknot-Legendrian}), $f(L_1)$ is Legendrian isotopic to $L_1$. By the contact isotopy extension theorem \cite[Theorem~2.6.2]{Geiges:book}, we can assume $f$ fixes $L_1$. Moreover, by Lemma~\ref{lem:fix}, we can further assume that $f$ fixes a standard neighborhood $S^0(-1)$ of $L_1$. As discussed above, the contact structure $\xi_{std}$ restricted to the complement of $S^0(-1)$, which is $S_{-p/q}(-1)$, is universally tight. Thus by Theorem~\ref{thm:cmcg-s1d2}, $f|_{S_{-p/q}(-1)}$ is contact isotopic to the identity relative to the boundary. Since $f|_{S^0(-1)}$ is the identity, $f$ is contact isotopic to the identity. This completes the proof of the claim. 

  Now assume $p=2$. Then $\pi_0(\Diff_+(L(p,q)))$ is trivial by Theorem~\ref{thm:mcg-lens}. In this case, any contactomorphism $f$ is contact isotopic to the identity, so we have
  \[
    \pi_0(\Cont(L(p,q),\xi_{std})) = 1. 
  \] 
  
  If $p\neq2$ and $q\equiv-1 \pmod p$, then $\pi_0(\Diff_+(L(p,q)))$ is generated by $\sigma$ by Theorem~\ref{thm:mcg-lens}. In this case, any contactomorphism $f$ is contact isotopic to $\sigma$ or the identity, so we have
  \[
    \pi_0(\Cont(L(p,q),\xi_{std})) = \mathbb{Z}_2,
  \] 
  generated by $\sigma$. 
  
  If $p\neq2$ and $q\equiv1 \pmod p$, then $\pi_0(\Diff_+(L(p,q)))$ is generated by $\tau$ by Theorem~\ref{thm:mcg-lens}. By Claim~\ref{claim:contacto}, we have
  \[
    \pi_0(\Cont(L(p,q),\xi_{std}))=1.
  \] 
 
  If $p\neq2$, $q\not\equiv\pm1$ and $q^2 \equiv 1 \pmod p$, then $\pi_0(\Diff_+(L(p,q)))$ is generated by $\sigma$ and $\tau$ by Theorem~\ref{thm:mcg-lens}. Again, any diffeomorphism $f$ which is smoothly isotopic to $\tau$ cannot be a contactomorphism of $\xi_{std}$ by Claim~\ref{claim:contacto}. Thus we have 
  \[
    \pi_0(\Cont(L(p,q),\xi_{std}))=\mathbb{Z}_2,
  \] 
  generated by $\sigma$. 
  
  Finally, if $q^2 \not\equiv 1 \pmod p$, then $\pi_0(\Diff_+(L(p,q)))$ is generated by $\tau$ by Theorem~\ref{thm:mcg-lens}. Again, any diffeomorphism $f$ smoothly isotopic to $\tau$ cannot be a contactomorphism of $\xi_{std}$ by Claim~\ref{claim:contacto}. Thus we have 
  \[
    \pi_0(\Cont(L(p,q),\xi_{std}))=1.
  \]
  This completes the proof.
\end{proof}

\begin{proof}[Proof of Corollary~\ref{cor:inclusion}]
  In the proof of Theorem~\ref{thm:cmcg-lens}, we showed that any contactomorphism $f \in \Cont(L(p,q),\xi_{std})$ is contact isotopic to either $\sigma$, $\overline\tau$, or the identity. This proves the injectivity of $i_*$ since $\pi_0(\Diff_+(L(p,q)))$ is generated by $\sigma$ and $\tau$.

  Also, in the proof of Theorem~\ref{thm:cmcg-lens}, we showed that any diffeomorphism which is smoothly isotopic to $\tau$ cannot be a contactomorphism of $\xi_{std}$ when $q \not\equiv -1 \pmod p$. Moreover, $\tau$ is not isotopic to the identity if $p\neq2$ by Theorem~\ref{thm:mcg-lens-kernel}. This implies that $i_*$ is not surjective when $q \not\equiv -1 \pmod p$. 
  
  Finally, in the proof of Theorem~\ref{thm:cmcg-lens}, we showed that both $\pi_0(\Cont(L(p,q),\xi_{std}))$ and $\pi_0(\Diff_+(L(p,q)))$ are generated by $\sigma$ if $q\equiv-1$, so $i_*$ is surjective.
\end{proof}

\begin{proof}[Proof of Theorem~\ref{thm:cmcg-s1s2}]
  Suppose $L$ is a Legendrian representative of the positively oriented core of $S^1 \times S^2$ with $\rot(L)=0$. First, observe that $\delta^m$ is not contact isotopic to $\delta^n$ if $m \neq n$ since $\rot(\delta^m(L)) \neq \rot(\delta^n(L))$ by Lemma~\ref{lem:rot-s1s2}. Let $f$ be a contactomorphism of $\xi_{std}$ on $S^1 \times S^2$. We will show that $\delta$ and $\eta$ commute, and $f$ is contact isotopic to $\delta^m \circ \eta^i$ for some $m \in \mathbb{Z}$ and $i \in \mathbb{Z}_2$. Then a map 
  \begin{align*}
    \Phi: \mathbb{Z} \oplus \mathbb{Z}_2 &\to \pi_0(\Cont(S^1 \times S^2), \xi_{std}),\\
    (m,i) &\mapsto \delta^m\circ\eta^i  
  \end{align*} 
  is a well-defined homomorphism and clearly it is an isomorphism. Suppose $n := \rot(f(L))$. There are two cases we need to consider according to the orientation of $f(L)$.

  Suppose $f(L)$ is smoothly isotopic to $L$. According to Lemma~\ref{lem:rot-s1s2}, $\delta$ increases the rotation number of $f(L)$ by $1$, so we have 
  \[
    \rot((\delta^{-n} \circ f)(L)) = 0.
  \]
  Thus $(\delta^{-n} \circ f)(L)$ is Legendrian isotopic to $L$ by Theorem~\ref{thm:unknot-s1s2}. Moreover, by Lemma~\ref{lem:fix}, we can assume that $\delta^{-n} \circ f$ fixes a standard neighborhood $N$ of $L$. Ding and Geiges \cite{DG:S1xS2} showed that the complement of $N$ in the standard contact structure on $S^1 \times S^2$ is a solid torus with two longitudinal dividing curves. Now by Theorem~\ref{thm:cmcg-s1d2}, the restriction of $\delta^{-n} \circ f$ to the complement of $N$ is contact isotopic to the identity. Thus $\delta^{-n} \circ f$ is contact isotopic to the identity and hence $f$ is contact isotopic to $\delta^n$. 
  
  Now suppose $f(L)$ is smoothly isotopic to $-L$. Then $(\eta^{-1} \circ \delta^n \circ f)(L)$ is smoothly isotopic to $L$ and by Lemma~\ref{lem:rot-s1s2} we have
  \[
    \rot((\eta^{-1} \circ \delta^n \circ f)(L)) = 0.
  \] 
  Thus by Theorem~\ref{thm:unknot-s1s2}, $(\eta^{-1} \circ \delta^n \circ f)(L)$ is Legendrian isotopic to $L$ and by Lemma~\ref{lem:fix}, we can assume that $\eta^{-1} \circ \delta^n \circ f$ fixes a neighborhood of $L$. Again, $\eta^{-1} \circ \delta^n \circ f$ is contact isotopic to the identity by the same argument above. Thus $f$ is contact isotopic to $\delta^{-n} \circ \eta$. 
  
  Finally, consider a contactomorphism $\delta^{-1} \circ \eta^{-1} \circ \delta \circ \eta$. By Lemma~\ref{lem:rot-s1s2}, we have
  \[
    \rot((\delta^{-1} \circ \eta^{-1} \circ \delta \circ \eta)(L)) = 0
  \]
  and $(\delta^{-1} \circ \eta^{-1} \circ \delta \circ \eta)(L)$ is smoothly isotopic to $L$. By Theorem~\ref{thm:unknot-s1s2}, $(\delta^{-1} \circ \eta^{-1} \circ \delta \circ \eta)(L)$ is Legendrian isotopic to $L$. Applying the argument above, we can show $\delta \circ \eta$ is contact isotopic to $\eta \circ \delta$. This completes the proof.   
\end{proof}

\begin{proof}[Proof of Corollary~\ref{cor:cont0}]
  Notice that $\Cont_0(M,\xi) = \ker i$, where $i:\Cont(M,\xi) \to \Diff_+(M)$ is the natural inclusion. Now the corollary is immediate from the injectivity of $i_*$ from Corollary~\ref{cor:inclusion}.
\end{proof}

\bibliography{references}
\bibliographystyle{plain}
\end{document}